\documentclass[11pt]{amsart}
\usepackage{amsmath,amscd}
\usepackage{graphicx, color}
\usepackage{amssymb, amsmath,amsthm}
\usepackage{epstopdf}
\newtheorem{thm}{{\bf Theorem}}[section]
\newtheorem{lem}[thm]{{\bf Lemma}}
\newtheorem{cor}[thm]{{\bf Corollary}}
\newtheorem{prop}[thm]{{\bf Proposition}}

\newtheorem{rem}[thm]{Remark}
\newtheorem{ex}[thm]{Example}

\newtheorem{ques}[thm]{Question}

\newtheorem{convention}[thm]{Convention}
\newtheorem{definition}[thm]{Definition}
\numberwithin{equation}{section}
   
\def\C{{\mathbb C}}

\def\Z{{\mathbb Z}}
\def\P{{\mathcal P}}

\def\T{{\mathbb T}}

\def\a{{\alpha}}
\def\b{{\beta}}
\def\c{{\gamma}}
\def\s{{\sigma}}

\begin{document} 

\title{Braids, orderings and minimal volume cusped hyperbolic 3-manifolds}

\date{2016/10/09}


\author[E. Kin]{Eiko Kin}

\address{%
       Department of Mathematics, Graduate School of Science, Osaka University Toyonaka, Osaka 560-0043, JAPAN
}
\email{%
        kin@math.sci.osaka-u.ac.jp
}

\author[D. Rolfsen]{Dale Rolfsen}

\address{Pacific Institute for the Mathematical Sciences and Department of Mathematics, 
University of British Columbia, Vancouver, BC, Canada V6T 1Z2}

\email{rolfsen@math.ubc.ca}



\thanks{EK was supported by 
Grant-in-Aid for
Scientific Research (C) (No. JP15K04875), 
Japan Society for the Promotion of Science.  DR acknowledges support of the Canadian Natural Sciences and Engineering Research Council.} 

\begin{abstract}  It is well-known that there is a faithful representation of braid groups on automorphism groups of free groups, and it is also well-known that free groups are bi-orderable.
We investigate which $n$-strand braids give rise to automorphisms which preserve some bi-ordering of the free group $F_n$ of rank $n$.  
As a consequence of our work we find that of the two minimal volume hyperbolic $2$-cusped orientable $3$-manifolds, one has bi-orderable fundamental group whereas the other does not.  
We prove a similar result for the $1$-cusped case, and have further results for more cusps.  
In addition, we study pseudo-Anosov braids and find that typically those with minimal dilatation are not order-preserving. 
\end{abstract}

\maketitle

\section{Introduction}\label{introduction}

If $<$ is a strict total ordering of the elements of a group $G$ such that $g < h$ implies $fg < fh$ for all $f, g, h \in G$, we call $(G,<)$ a {\em left-ordered} group.  If the left-ordering $<$ is also invariant under right-multiplication, we call 
$(G,<)$ a {\em bi-ordered} group (sometimes known in the literature simply as ``ordered'' group).  If a group admits such an ordering it is said to be {\em left-} or {\em bi-orderable}.  
It is easy to see that a group is left-orderable if and only if it is right-orderable.   
If $(G, <)$ is a bi-ordered group, then $<$  is invariant under conjugation:
$g< h$ if and only if $fgf^{-1} < fhf^{-1}$ for all $f,g,h \in G$.  Nontrivial examples of bi-orderable groups are the free groups $F_n$ of rank $n$, as discussed in Appendix \ref{Ordering free groups}.

An automorphism $\phi$ of a group $G$ is said to {\it preserve an ordering} $<$ of $G$ if for every $f, g \in G$ we have $f < g \implies \phi(f) < \phi(g)$;  we also say $<$ is {\it $\phi$-invariant}.   
If $\phi: G \rightarrow G$ preserves an ordering $<$ of $G$, 
then the $n$th power $\phi^n$ preserves the ordering $<$ of $G$ for each integer $n$.
We note that if $\phi^n$ preserves an ordering it does not necessarily follow that $\phi$ does.

E. Artin \cite{Artin25, Artin47} observed that each $n$-strand braid corresponds to an automorphism of $F_n$.   This paper concerns the question of which braids give rise to automorphisms which preserve some bi-ordering of $F_n$.  In turn this is related to the orderability of the fundamental group of the complement of certain links in $S^3$, namely the braid closure together with its axis, which we call a braided link. 
 We pay special attention to pseudo-Anosov mapping classes and their stretch factors (dilatations), and cusped hyperbolic $3$-manifolds of small volume.  

This paper is organized as follows: Section~\ref{Braids and $Aut(F_n)$} reviews the relation among braids, mapping class groups, free group automorphisms and certain links in the 3-sphere $S^3$.  In 
Section~\ref{Orderable groups} we recall basic properties of orderable groups, with explicit bi-ordering of free groups further described in Appendix~\ref{Ordering free groups}.  The study of order-preserving braids and their relation to bi-ordering the group of the corresponding braided links is initiated in 
Section~\ref{section_OPbraids}.  Applications to cusped hyperbolic 3-manifolds of minimal volume are considered in Section~\ref{section_hyperbolic}.  
In Section~\ref{section_nonOP} we give many examples of non-order-preserving braids, including pseudo-Anosov braids with minimal dilation as well as large dilatations.  
We also find a family of pretzel links whose fundamental groups can not be bi-orderable. 
Appendix~\ref{Whitehead} is devoted to a proof that the fundamental group of the Whitehead link complement is bi-orderable.
\medskip

\noindent
{\bf Acknowledgments.} 
We thank Tetsuya Ito for helpful conversations and comments.

\section{Braids and $\mathrm{Aut}(F_n)$} \label{Braids and $Aut(F_n)$} Let $B_n$ be the $n$-strand braid group, which has the well-known presentation
with generators $\s_1, \dots , \s_{n-1}$ subject to the relations $\s_i\s_j = \s_j\s_i$ if $| i - j | >1$ and $\s_j\s_i\s_j = \s_i\s_j\s_i$ if 
$| i - j | =1$. See Figure~\ref{fig_halftwist}(1)(2).

\subsection{Mapping classes}
Let $D_n$ denote the disk with $n$ punctures, which we may picture as equally spaced along a diameter of the disk and labelled $1$ to $n$.  $\mathrm{Mod}(D_n)$ denotes the mapping class group of $D_n$ and $\mathrm{Mod}(D_n, \partial D)$ the mapping class group of homeomorphisms fixed on the boundary pointwise.  
There is a well-known isomorphism 
$$\bar{\Gamma}:  B_n \rightarrow \mathrm{Mod}(D_n, \partial D)$$ 
which sends $\s_i$ to a half twist $h_i$ which interchanges the punctures labelled $i$ and $i+1$, 
see Figure~\ref{fig_halftwist}(3).  
The kernel of the obvious map $\mathrm{Mod}(D_n, \partial D) \to  \mathrm{Mod}(D_n)$ is infinite cyclic, generated by a Dehn twist along a simple closed curve parallel to the boundary of the disk. 
Using the isomorphism $ \bar{\Gamma}: B_n \rightarrow \mathrm{Mod}(D_n, \partial D)$ together with this obvious map, 
we have the surjective homomorphism 
$$\Gamma: B_n  \rightarrow \mathrm{Mod}(D_n)$$ 
whose kernel is generated by the full twist  $\Delta_n^2 \in B_n$, 
where $\Delta_n \in B_n$ is the half twist.

Elements of $\mathrm{Mod}(D_n)$ are classified into three types: 
periodic, reducible and pseudo-Anosov, called Nielsen-Thurston types \cite{Thurston88}.   
If two given mapping classes are conjugate to each other, then their Nielsen-Thurston types are the same. 
We say that $\b \in B_n$ is {\it periodic} (resp. {\it reducible}, {\it pseudo-Anosov}) if its  mapping class $\Gamma(\beta) \in  \mathrm{Mod}(D_n)$ is of the corresponding type.

\subsection{Free group automorphisms}

Let $\beta $ be an $n$-strand braid.  
 Let $\phi: D_n \rightarrow D_n$ be a 
representative of the mapping class $\bar{\Gamma}(\beta) \in \mathrm{Mod}(D_n, \partial D)$. 
Obviously $\phi$ represents a mapping class $\Gamma(\beta) \in \mathrm{Mod}(D_n)$. 
If one passes to the induced map $\phi_*= \phi_{*p}: \pi_1(D_n,p) \rightarrow \pi_1(D_n,p)$ 
of the fundamental group of $D_n$, using a point $p$ on the boundary as basepoint, this defines the Artin representation 
$$B_n \to \mathrm{Aut}(F_n)$$ 
which can be defined on the generators as follows, 
where 
$x_1, x_2 \dots, x_n$ are the free generators of the free group $F_n$ of rank $n$ and 
$\mathrm{Aut}(F_n)$ is the group of automorphisms of $F_n$.  
The generator $\s_i$ induces the automorphism
\begin{equation}
\label{equation_Artin}
x_i \mapsto x_ix_{i+1}x_i^{-1}, \quad   x_{i+1} \mapsto x_i,  \quad  x_j \mapsto x_j \;{\rm if}\; j \ne i,i+1.
\end{equation}
See Figure~\ref{fig_Artin-i}. 
It is known that the Artin representation is faithful.  
Its image is the subgroup of automorphisms of $F_n$ that take each $x_i$ to a conjugation of some $x_j$ and which take the product $x_1x_2 \cdots x_n$ to itself.

We note that if $\Gamma(\beta) = \Gamma(\beta') \in \mathrm{Mod}(D_n)$ for $n$-strand braids $\beta$ and $\beta'$, 
then $\beta' = \beta \Delta_n^{2k}$ for some integer $k$. 
The images of $\beta$ and $\beta \Delta_n^{2k}$ under the Artin representation are the same up to an inner automorphism 
$$x \rightarrow (x_1 x_2 \cdots x_n)^{k} x (x_1 x_2 \cdots x_n)^{-k}.$$

By abuse of notation, from now on we will use the same symbol $\b$ for the braid, the mapping class 
$\Gamma(\beta) \in \mathrm{Mod}(D_n)$ and the corresponding automorphism of $F_n$.

An automorphism $\phi$ of $F_n = \langle x_1, \dots, x_n \rangle$ is said to be {\em symmetric} if for each generator $x_j$, the image $\phi(x_j)$ is a conjugate of some $x_k$.  
Every automorphism on $F_n$ corresponding to the action of a braid on $F_n$ is symmetric. 
A symmetric automorphism $\phi: F_n \rightarrow F_n$ is {\em pure} if $\phi$ sends each $x_i$ to a conjugate of itself.  
Pure braids (see Section~\ref{subsection_Pure-braids}) induce symmetric and pure automorphisms.

Since braid words, like paths, are typically read from left to right, we adopt the convention that braids act on $F_n$ on the right.  
If $x \in F_n$, we denote the action of $\b \in B_n$ by $x \to x^\b$, and if $\b, \c \in B_n$ we have the identity 
$x^{\b\c} = (x^\b)^\c$.

\begin{definition}
An $n$-strand braid $\b$ is said to be order-preserving if there exists some bi-ordering $<$ of $F_n$ preserved by the automorphism $x \to x^\b$ of $F_n$.
\end{definition}

One sees that 
$\b \in B_n$ is order-preserving if and only if $\b\Delta_n^{2k}$ is order-preserving for some (hence all) $k \in \Z$ (Corollary~\ref{cor_multDelta^2}).

As discussed below, there is some ambiguity in defining the action of $B_n$ on $F_n$, depending on choices of a representative of the mapping class $ \beta \in \mathrm{Mod}(D_n)$, 
basepoint in $D_n$ and generators of $\pi_1(D_n)$.  As we'll see, this ambiguity is irrelevant in the question of whether a braid $\beta$ is order-preserving.

\subsection{Basepoints}  
It is sometimes convenient to use a basepoint and generators different from that used in the Artin representation.  Specifically, we consider a representative $\phi: D_n \to D_n$ 
of a mapping class in $\mathrm{Mod}(D_n, \partial D)$ and we may assume that 
$p, q \in D_n$ are two different points (not necessarily on the boundary of the disk), each fixed by $\phi$.  
Then we have induced maps
$\phi_{*p}: \pi_1(D_n, p) \to \pi_1(D_n, p)$ and $\phi_{*q}: \pi_1(D_n, q) \to \pi_1(D_n, q)$.  Of course,
$\pi_1(D_n, p)$ and $\pi_1(D_n, q)$ are isomorphic, but not canonically.  
We can construct an isomorphism by choosing a path $\ell$ in $D_n$ from $q$ to $p$ which then defines an isomorphism
$h: \pi_1(D_n, p) \to \pi_1(D_n, q)$ sending the class of a loop $\a$ in $D_n$ based at $p$ to the class of the loop $\ell\a\ell^{-1}$ based at $q$.  Consider the diagram 
$$\begin{CD}
\pi_1(D_n, p)         @>{\phi_{*p}}>>        \pi_1(D_n, p) \\
@V{h}VV                                                   @V{h}VV   \\
\pi_1(D_n, q)         @>{\phi_{*q}}>>        \pi_1(D_n, q)
\end{CD}$$
which is not necessarily commutative.
However, we leave it to the reader to check the following
\begin{prop}
The above diagram commutes up to conjugation.  Specifically, $h \circ \phi_{*p}$ equals
$\phi_{*q} \circ h$ followed by a conjugation in $\pi_1(D_n, q)$, the conjugating element being the class of the loop $\ell(\phi \circ \ell)^{-1}$ in $\pi_1(D_n, q)$.
\end{prop}

\begin{cor}
\label{cor_well-defined}
The map $\phi_{*p}$ preserves a bi-ordering of $\pi_1(D_n, p) $ if and only if $\phi_{*q}$ preserves a bi-ordering of $\pi_1(D_n, q).$
\end{cor}

\begin{proof}
If $\phi_{*p}$ preserves the bi-ordering $<_p$ of $\pi_1(D_n, p)$, define a bi-ordering $<_q$ of 
$\pi_1(D_n, q)$ by the formula $f <_q g \iff h^{-1}(f) <_p h^{-1}(g)$ for $f, g \in \pi_1(D_n, q)$.  Then
one checks that $f<_q g \implies \phi_{*q}(f) <_q \phi_{*q}(g)$ using conjugation invariance of bi-orderings.  The converse is proved similarly.
\end{proof}

If one passes from $\mathrm{Mod}(D_n, \partial D)$ to $\mathrm{Mod}(D_n)$, there is a further ambiguity regarding the action of a braid on $\pi_1(D_n)$.  However, this ambiguity corresponds to conjugation by a power of $\Delta_n^2$, so again it is irrelevant to the question of preserving a bi-ordering. 
More concretely,   
given an $n$-strand braid $\beta$, let $\phi: D_n \rightarrow D_n$ be any representative of the mapping class $\beta \in \mathrm{Mod}(D_n)$. 
We take any basepoint $q$ of $D_n$ possibly $\phi(q) \ne q$. 
By choosing a path $\ell$ in $D_n$ from $q$ to $\phi(q)$, 
we have the induced map 
\begin{equation}
\label{equation_induced-map}
\phi_{*q}: \pi_1(D_n,q) \rightarrow  \pi_1(D_n,q)
\end{equation}
which sends the class of a loop $\alpha$ in $D_n$ based at $q$ 
to the class of the loop $\ell \alpha \ell^{-1}$ based at the same point. 
By using  Corollary~\ref{cor_well-defined}, one sees that 
 $\beta$ is order-preserving if and only if
$\phi_{*q}: \pi_1(D_n,q) \rightarrow \pi(D_n,q)$ 
preserves a bi-ordering of $\pi_1(D_n,q)$. 

By abuse of notation again, 
we denote the induced map $\phi_{*q}$ by $\beta$, when $q$ is specified. 

These remarks show that if one allows different choices of basepoint, say a basepoint $q$, 
the action of $B_n$ on $F_n$ should really be regarded as a representation 
$$B_n \to   \mathrm{Out}(\pi_1(D_n,q))  \cong \mathrm{Out}(F_n),$$ 
the group of outer automorphisms.  
Recall that  $\mathrm{Out}(G) = \mathrm{Aut}(G)/\mathrm{Inn}(G)$, where 
$\mathrm{Inn}(G)$ is the (normal) subgroup of inner automorphisms of a group $G$.

%


\subsection{Mapping tori and braided links}

For a braid $\b \in B_n$, we denote the mapping torus by 
$$\T_{\b} = D_n \times [0, 1] / (y, 1) \sim (\phi(y), 0),$$ 
where $\phi: D_n \rightarrow D_n$ is a representative of $ \beta \in  \mathrm{Mod}(D_n)$. 
By the hyperbolization theorem of Thurston~\cite{Thurston98}, 
$\T_\b$ is hyperbolic if and only if $\b$ is pseudo-Anosov.

The closure $\widehat{\b}$ of a braid $\beta$ is a knot or link in the $3$-sphere $S^3$ and the {\em braided link}, denote by $\mathrm{br}(\b)$, is the closure $\widehat{\b}$, 
together with the braid axis $A$, which is an unknotted curve that $\widehat{\b}$ runs around in a monotone manner: 
$\mathrm{br}(\b) = \widehat{\b} \cup A$. 
See Figure~\ref{fig_braid_axis}(1)(2).  
Whereas all links can be realized as $\widehat{\b}$, this is not true for $\mathrm{br}(\b)$, as each component of 
$\widehat{\b}$ has nonzero linking number with the braid axis $A$.  As an example the Whitehead link considered in Appendix~\ref{Whitehead} is not a braided link.

\begin{center}
\begin{figure}
\includegraphics[width=3.5in]{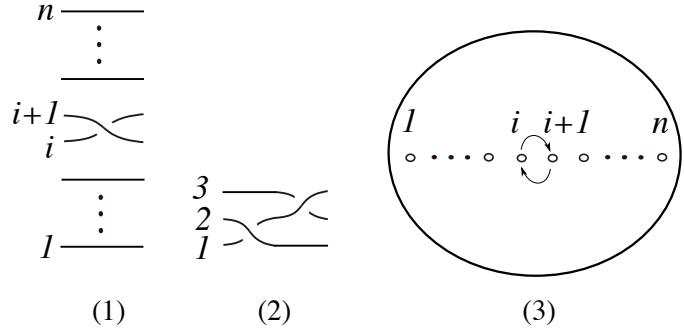}
\caption{(1) $\sigma_i \in B_n$.  
(2) $\sigma_1 \sigma_2^{-1} \in B_3$. 
(3)  $h_i \in \mathrm{Mod}(D_n, \partial D)$.}
\label{fig_halftwist}
\end{figure}
\end{center}

\begin{center}
\begin{figure}
\includegraphics[width=5in]{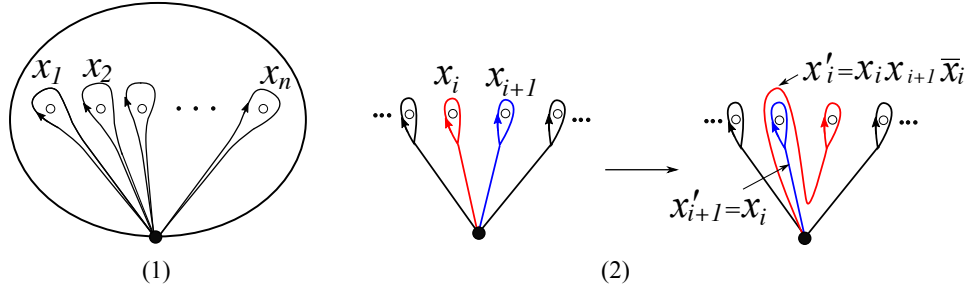}
\caption{A basepoint $\bullet$  of $\pi_1(D_n)$ lies on $\partial D$. 
(1) Generators $x_i$'s of $F_n$. 
(2) $\sigma_i: F_n \rightarrow F_n$, 
where $x':=$ the image of $x$ under $\sigma_i$ and 
$\overline{x}:= x^{-1}$.} 
\label{fig_Artin-i}
\end{figure}
\end{center}

\begin{center}
 \begin{figure}
\includegraphics[width=4.5in]{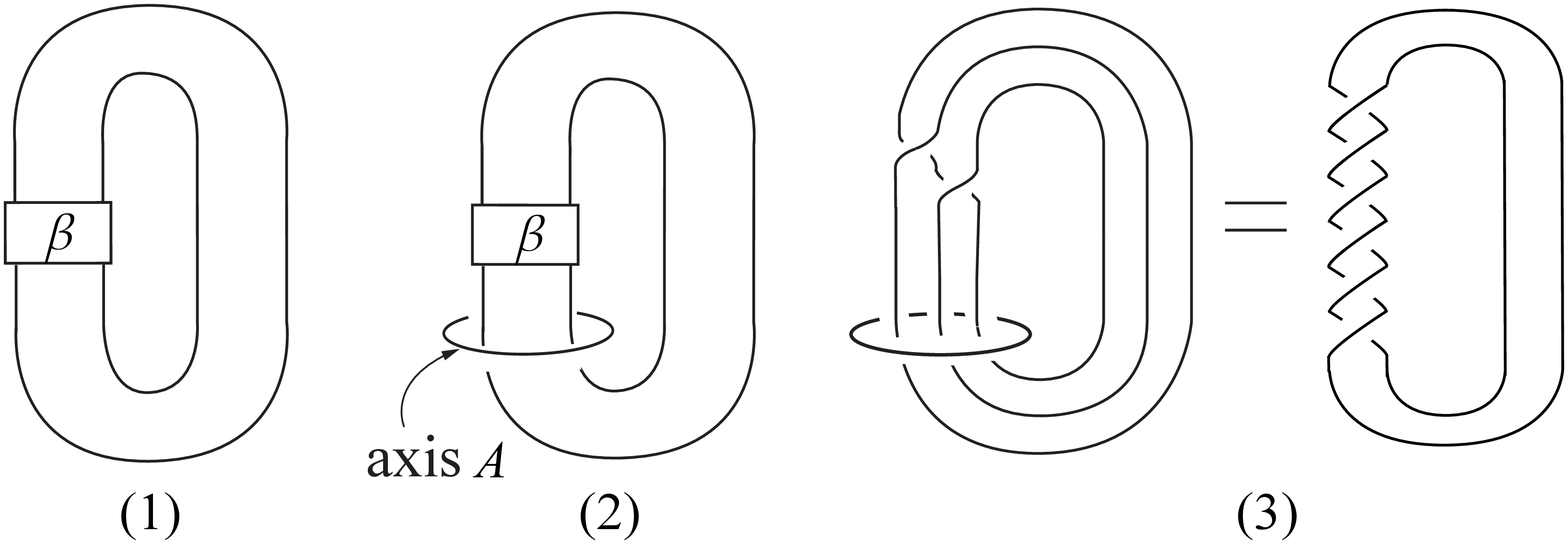}
\caption{(1) Closure $\widehat{\beta}$. (2) $\mathrm{br}(\beta) = \widehat{\beta} \cup A$. 
(3) $\mathrm{br}(\sigma_1 \sigma_2)$ is equivalent to the $(6,2)$-torus link.}
\label{fig_braid_axis}
\end{figure}
\end{center}

We see that  $\T_\b$ is homeomorphic with the complement of the braid closure 
$\widehat{\b}$ in the solid torus 
$D^2 \times S^1$.  The interior $Int(\T_\b)$ can be identified with the complement of $\widehat{\b} \cup A$ in $S^3$, so we have the following.


\begin{lem}\label{lemma}
$Int(\T_\b)$ is homeomorphic to $S^3 \setminus \mathrm{br}(\b)$.
\end{lem}

Of course the fundamental group of $Int(\T_\b)$  is isomorphic with that of $\T_\b$, which in turn is the semidirect product 
$$\pi_1(S^3 \setminus br(\b)) \cong \pi_1(\T_\b) \cong F_n \rtimes_\b \Z.$$   
There is a (split) exact sequence 
$$1 \rightarrow F_n \rightarrow \pi_1(\T_\b) \rightarrow \Z = \langle t \rangle \rightarrow 1.$$
\noindent
One may consider $F_n \rtimes_\b \Z$ as the set of ordered pairs 
$$F_n \rtimes_\b \Z = \{(f, t^p)\ |\ f \in F_n, \ p \in \Z\},$$ with the multiplication given by 
$$(f, t^p)(g, t^q) = (ft^pgt^{-p}, t^{p+q}), \quad {\rm where} \quad tgt^{-1} = g^{\b}.$$

\section{Orderable groups} \label{Orderable groups}
In this section we  outline a few well-known facts about orderable groups.  For more details, see \cite{CR15,MR77}.  Further details regarding bi-ordering of $F_n$ may be found in Appendix \ref{Ordering free groups}.

\begin{prop} \label{positive cone} A group $G$ is left-orderable if and only if there exists a sub-semigroup $\P \subset G$ such that for every 
$g \in G$ exactly one of $g = 1$, $g \in \P$ or $g^{-1} \in \P$ holds.  
\end{prop}

Indeed such a $\P$ defines a left-ordering $< \; = \; <_\P$ by the rule $g < h \iff g^{-1}h \in \P$.  
Conversely, a left-ordering $<$ defines a {\it positive cone} $\P = \P_< := \{g \in G \mid 1 < g \}$ satisfying the conditions of Proposition~\ref{positive cone}.

\begin{prop} 
\label{prop_BO-criterion}
A group $G$ is bi-orderable if and only if it posesses $\P \subset G$ as in Proposition \ref{positive cone}, and in addition $g^{-1}{\P}g = \P$ for all $g \in G$.
\end{prop}

We note that a left- or bi-ordering of $G$ is preserved by an automorphism $\phi:G \to G$ if and only if $\phi(\P) = \P$, where $\P$ is the positive cone of the ordering.

\begin{prop} 
\label{prop_extension-LO}
Suppose $1 \rightarrow K  \stackrel{i}{\hookrightarrow} G \stackrel{p}{\rightarrow} H \rightarrow 1$ is an exact sequence of groups.  If $K$ and $H$ are left-ordered with positive cones $\P_K$ and $\P_H$ respectively, then $G$ is left-orderable using the positive cone $\P_G := i(\P_K) \cup p^{-1}(\P_H)$.  
\end{prop}

This is sometimes called the {\em lexicographic} order of $G$, considered as an extension.

\begin{prop} 
\label{prop_extension}
    In Proposition~\ref{prop_extension-LO}, 
    if $K$ and $H$ are bi-ordered, then the formula $\P_G = i(\P_K) \cup p^{-1}(\P_H)$ defines a bi-ordering of $G$ if and only if the bi-ordering of $K$ is respected under conjugation by elements of $G$; equivalently $g^{-1}{\P_K}g = \P_K$ for all $g \in G$. 
\end{prop}

A subset $S$ of a left-ordered group $(G, <)$ is said to be {\em convex} 
if for all $s, s' \in S$ and $g \in G$ satisfying $s < g < s'$, we have $g \in S$.

\begin{prop}\label{convex kernel} If  $1 \rightarrow K  \stackrel{i}{\rightarrow} G \stackrel{p}{\rightarrow} H \rightarrow 1$ is an exact sequence of groups 
with $(G, <)$ a left-ordered group and $K$ a convex subgroup of  $(G, <)$, then there is a left-ordering $\prec$ of $H$ in which
$1 \prec h$ if and only if some (and hence every) element $g \in p^{-1}(h)$ satisfies $1 < g$.  If $(G, <)$ is a bi-ordered group, so is $(H, \prec)$.
\end{prop}

Being bi-orderable is a much stronger property than being left-orderable.  An intermediate property of groups is to be {\em locally-indicable}, which means that every finitely-generated nontrivial subgroup has the infinite cyclic group $\Z$ as a quotient. 

\begin{prop}  For a group, the following implications hold: bi-orderable $\implies$ locally-indicable $\implies$ left-orderable $\implies$ torsion-free.  None of these implications is reversible.
\end{prop} 

\begin{prop}[\cite{BRW05}] If $L$ is a knot or link in $S^3$, then its group $\pi_1(S^3 \setminus L)$ is locally indicable and therefore left-orderable.
\end{prop} 

Certain knot groups and link groups are bi-orderable, and this is a particular focus of this paper.  For example, torus knot groups are not bi-orderable, but the figure-eight knot $4_1$ has bi-orderable group \cite{PR03}, and we shall see that the same is true of the Whitehead link and many links constructed from braids as above.
One of the reasons bi-orderability of knot groups is of interest has to do with surgery and $L$-spaces, which were introduced by Ozsv\'ath and Szab\'o \cite{OS05} and include all 3-manifolds with finite fundamental group.

\begin{thm}[\cite{CR12}]  \label{thm_L-space}
If $K$ is a knot in $S^3$ for which $\pi_1(S^3 \setminus K)$ is bi-orderable, then surgery on $K$ cannot produce an $L$-space.
\end{thm}

We note that this is not true in general for links.  As we shall see in Theorem~\ref{thm_Whitehead},  
the Whitehead link has bi-orderable fundamental group, but one can construct lens spaces by certain surgeries on the link.



Consider a group $K$, an automorphism $\phi: K \to K$ and the semidirect product 
$G = K \rtimes_\phi \Z $.  Recall $G \cong \{(k, t^p)\ |\ k \in K,\ p \in {\Bbb Z}\}$ with multiplication given by 
$$(k_1, t^p)(k_2, t^q) = (k_1t^pk_2t^{-p}, t^{p+q}) = (k_1\phi^p(k_2), t^{p+q}), \quad {\rm where} \quad tkt^{-1} = \phi(k).$$

\begin{prop}\label{HNN}
Suppose $K$ is bi-orderable and $\phi: K \to K$ is an automorphism.  Then $G = K \rtimes_\phi \Z$
is bi-orderable if and only if there exists a bi-ordering of $K$ which is preserved by $\phi$.
\end{prop}

\begin{proof}
Suppose that $G$ is bi-ordered by $<$.   Since bi-orderings are invariant under conjugation, the equation $\phi(k) = tkt^{-1}$ implies that $<$, restricted to $K$, is $\phi$-invariant.
For the converse, suppose $\phi$ preserves a bi-ordering $\prec$ of $K$.  Then we use the lexicographic ordering defined by $(k_1, t^p) < (k_2, t^q)$ if and only if $p < q$ as integers or $p=q$ and $k_1 \prec k_2$, to order $G$.  It is easily checked that this is a bi-ordering using the identity $t^pkt^{-p} = \phi^p(k)$ and the assumption that $\phi$ preserves the ordering $\prec$.
\end{proof}

Let $\Sigma$ be an orientable surface. 
Then  $\pi_1(\Sigma)$ is bi-orderable, see \cite{RW01}. 
Let $\mathrm{Mod}(\Sigma)$  be the mapping class group of $\Sigma$, let $f = [\phi] \in \mathrm{Mod}(\Sigma)$ and assume $\phi(p) = p$ for some $p \in \Sigma$.
Since the fundamental group of the mapping torus ${\Bbb T}_{f}$ of $f$ is 
the semidirect product $\pi_1(\Sigma) \rtimes_{\phi_*} {\Bbb Z}$, 
where $\phi_*: \pi_1(\Sigma, p) \rightarrow \pi_1(\Sigma, p)$ is an automorphism induced from $f$, 
we have the following from Proposition~\ref{HNN}.

\begin{prop}
\label{prop_fibration-HNN}
 $\pi_1({\Bbb T}_{f})$ is bi-orderable if and only if 
there exists a bi-ordering of $\pi_1(\Sigma, p)$ which is preserved by $\phi_*: \pi_1(\Sigma,p) \rightarrow  \pi_1(\Sigma,p)$, 
where $\phi$ is a representative of the mapping class $f$. 
\end{prop}

If $\phi(x)= x$ for $x \in G$, then 
we say that the orbit of $x$ (under $\phi$) is {\it trivial}.

\begin{prop}\label{finite orbit}
Let $G$ be a left-orderable group. 
If an automorphism $\phi: G \to G$ preserves a left-ordering $<$ of $G$, then $\phi$ cannot have any nontrivial finite orbits.
\end{prop}

\begin{proof}
If the orbit of $x \in G$ is nontrivial, we may assume $x < \phi(x)$.  Then $\phi(x) < \phi^2(x)$, and by transitivity and induction $x < \phi^m(x)$ for all $m \ge 1$ and 
therefore the orbit of $x$, $\{\phi^n(x)\ |\ n \in \Z\}$, is infinite.
\end{proof}

\noindent
Proposition~\ref{finite orbit} says that 
if $\phi: G \rightarrow G$ has a nontrivial finite orbit, then $\phi$ does not preserve any left-ordering  of $G$. 

%
%

Consider an automorphism $\phi : F_n \to F_n$ of a free group. 
The following two criteria for $\phi$ being order-preserving (or not) will be useful; they involve the abelianization $\phi_{ab}: \Z^n \to \Z^n$ and its eigenvalues, which are, {\em a priori},
$n$ complex numbers, possibly with multiplicity.

\begin{thm}[\cite{PR03}]  \label{Perron-Rolfsen}
Let $\phi : F_n \to F_n$ be an automorphism. 
If every eigenvalue of $\phi_{ab}: \Z^n \to \Z^n$ is real and positive, then there is a bi-ordering of $F_n$ which is $\phi$-invariant.
\end{thm}

\begin{thm}[\cite{CR12}] \label{Clay-Rolfsen}
If there exists a bi-ordering of $F_n$ which is $\phi$-invariant, then $\phi_{ab}$ has at least one real and positive eigenvalue. 
\end{thm}

This is useful in showing that certain fibred $3$-manifolds have fundamental groups which are {\em not} bi-orderable.  However, we note that in the case of braids it does not apply in that way, for if $\phi$ is an automorphism of $F_n$ induced by a braid $\b \in B_n$, or more generally a symmetric automorphism, then $\phi_{ab}$ is simply a permutation of the generators of $\Z^n$ and therefore has at least one eigenvalue equal to one.

We note that Theorem \ref{Perron-Rolfsen} cannot have a full converse.  It has been observed in \cite{LRR08} that for $n \ge 3$ there exist automorphisms of $F_n$ which preserve a bi-ordering of $F_n$, but whose eigenvalues are precisely the 
$n$th roots of unity in $\C$.  Such examples appear in the discussion in Section \ref{explicit}. 

In Appendix \ref{Ordering free groups} a certain class of bi-orderings of the free group $F_n$, called {\it standard} orderings is defined, using the lower central series.


\begin{prop}\label{sym not id}
If $\phi: F_n \to F_n$ is a non-pure symmetric   automorphism, then $\phi$ cannot preserve any standard bi-ordering of $F_n$.
\end{prop} 


\begin{proof}
Assume $\phi: F_n \rightarrow F_n$ is not pure, but preserves a standard bi-ordering $<$ of $F_n$. 
Then $\phi_{ab}$  is a nontrivial permutation of the generators of the abelianization $\Z^n$.  
By Proposition
\ref{LC convex}, the commutator subgroup $[F_n, F_n]$ is convex relative to $<$, and therefore $F_n/[F_n, F_n] \cong \Z^n$ inherits a bi-ordering $<_1$ according to Proposition \ref{convex kernel}.  Since $\phi$ preserves $<$, one easily checks that $\phi_{ab}$ 
preserves the order $<_1$ of $\Z^n$.  But, $\phi$ being non-pure symmetric implies that $\phi_{ab}$ is a nontrivial permutation of the generators of $\Z^n$ which are the images of the $x_i$.  By Proposition \ref{finite orbit}, $\phi_{ab}$ cannot preserve any ordering of $\Z^n$, a contradiction.
\end{proof}

\begin{prop}
\label{prop_pure-symmetric}
If $\phi: F_n \rightarrow F_n$ is a pure symmetric automorphism, 
then $\phi$ is order-preserving.  In fact, it preserves every standard ordering of $F_n$.
\end{prop} 

\begin{proof}
A pure symmetric automorphism $\phi: F_n \rightarrow F_n$  induces the identity map 
$\phi_{ab} = id: {\Bbb Z}^n \rightarrow {\Bbb Z}^n$, so Proposition \ref{auto identity} applies.
\end{proof}

\noindent
For any symmetric automorphism $\phi: F_n \rightarrow F_n$, 
there exists  $k \ge 1$ so that $\phi^k$ is pure symmetric. 
By Proposition~\ref{prop_pure-symmetric}, we have the following. 

\begin{cor}
\label{cor_power}
If $\phi: F_n \rightarrow F_n$ is a symmetric automorphism, 
then there exists  $k \ge 1$ such that 
$\phi^k$ is order-preserving. 
\end{cor}

\section{Order-preserving braids}
\label{section_OPbraids}

\subsection{Basic properties} 
\label{subsection_Basic}

From Lemma \ref{lemma} and Proposition~\ref{prop_fibration-HNN} we have the following.

\begin{prop}\label{prop_OPiffBi-ord}
A braid $\b \in B_n$ is order-preserving if and only if $\pi_1(S^3 \setminus \mathrm{br}(\b))$ is bi-orderable.
\end{prop}

Let $\delta_n$ be the $n$-strand braid 
$$\delta_n=\sigma_1 \sigma_2 \cdots \sigma_{n-1} $$ and  
let $\Delta = \Delta_n \in B_n$ be the half twist 
$$\Delta = (\s_1\s_2\cdots \s_{n-1})(\s_1\s_2\cdots \s_{n-2}) \cdots (\s_1\s_2)\s_1.$$
The full twist  $\Delta^2$ is written by 
$\Delta^2 = \delta_n^n= (\delta_n \sigma_1)^{n-1}$, 
which means that $\delta_n$ and $\delta_n \sigma_1$ are roots of $\Delta^2$. 
We note that $\Delta^2 $ commutes with all $n$-strand braids and in fact generates the centre of $B_n$ when $n \ge 3$, 
see \cite[Theorem~1.24]{KT08}.

\begin{cor}\label{cor_multDelta^2}
A braid $\b \in B_n$ is order-preserving if and only if $\b\Delta^{2k}$ is order-preserving for some (hence all) $k \in \Z.$  Moreover, they preserve exactly the same bi-orderings of $F_n$.
\end{cor}

\begin{proof}
This follows since $S^3 \setminus \mathrm{br}(\b)$ and $S^3 \setminus \mathrm{br}(\b\Delta^{2k})$ are homeomorphic: 
By using the {\it disk twist} as we shall define in Section~\ref{subsection_DiskTwists}, we see that 
$k$th power of the disk twist $\mathfrak{t}^k$ about the disk bounded by the braid axis $A$ of $\beta$ 
sends the exterior $\mathcal{E}(\mathrm{br}(\beta))$ of the link $\mathrm{br}(\beta)$ to the exterior $\mathcal{E}(\mathrm{br}(\beta \Delta^{2k}))$. 
An alternative argument is that $\Delta^2$ acts by conjugation of $F_n$, so it preserves {\em every} bi-order of $F_n$.
\end{proof}

As noted by Garside \cite{Garside69} every $n$-strand braid has an expression $\b\Delta^{2k}$ where $\b$ is a {\em Garside positive} braid word, meaning that $\b$ can be written as a word in the $\s_i$ generators without negative exponents.  Thus a question of a braid being order-preserving can always be reduced to the case of positive braids.  
Notice that changing a braid $\b$ by conjugation does not change the link $\mathrm{br}(\b)$ up to isotopy, so we have:

\begin{cor}
Let $\alpha, \beta \in B_n$. 
Then $\b$ is order-preserving if and only if $\a\b\a^{-1}$ is order-preserving.
\end{cor}

The very simplest of nontrivial braids are the generators $\s_i $.

\begin{prop}  \label{gen not OP}
The generators $\s_i \in B_n$ are not order-preserving.
\end{prop}

\begin{proof}
Suppose that $\s_i$ preserves a bi-order $<$ of $F_n$.  We may assume $x_i < x_{i+1}$.  Then 
$x_i^{\s_i} < x_{i+1}^{\s_i}$, or in other words $x_ix_{i+1}x_i ^{-1} < x_i$, see (\ref{equation_Artin}).  
But conjugation invariance of the ordering then yields 
$x_{i+1}< x_i$, which is a contradiction. 

An alternative argument is as follows. 
We take a basepoint $p_1$  in the interior of $D_n$. 
There exists a homeomorphism $\phi: D_n \rightarrow D_n$ which represents $\sigma_i \in \mathrm{Mod}(D_n)$ 
such that the induced map 
$\phi_*: \pi_1(D_n, p_1) \rightarrow  \pi_1(D_n, p_1)$ 
gives rise to the following automorphism on $F_n= \langle x_1, \cdots, x_n \rangle$: 
\begin{equation}
\label{equation_sigma}
\sigma_i: x_i \mapsto x_{i+1}, \quad x_{i+1} \mapsto x_{i+1}x_i x_{i+1}^{-1}, \quad x_j \mapsto x_j\ {\rm if}\ 
j \ne i,i+1,
\end{equation}
see Figure~\ref{fig_sigma-i}(2). 
(In the figures, we denote by $x'$, the image of $x \in F_n$ under the automorphism of $F_n$, 
and denote by $\overline{x}$, the inverse $x^{-1}$ of $x$.)
Suppose that $\s_i$ preserves a bi-order $<$ of $F_n$.  
We may assume $x_i < x_{i+1}$. 
Then $x_{i+1} < x_{i+1}x_i x_{i+1}^{-1}$ 
since $\sigma_i$ preserves $<$. 
This implies that 
$x_{i+1}< x_i$ by conjugation invariance of the bi-ordering. 
This is a contradiction.  
\end{proof}

\begin{center}
\begin{figure}
\includegraphics[width=4in]{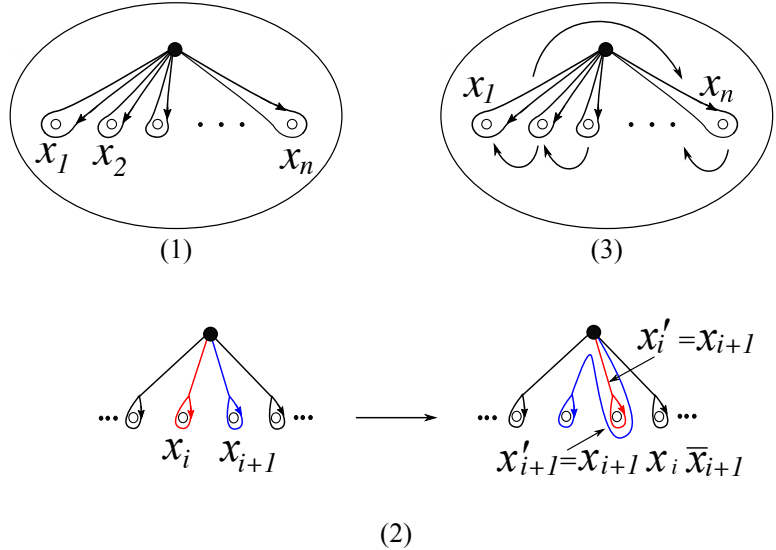}
\caption{A basepoint $\bullet$  of $\pi_1(D_n)$ is taken to be in the interior of $D_n$. 
(1) Generators $x_i$'s of $F_n $. 
(2) $\sigma_i: F_n \rightarrow F_n$.
(3) $\delta_n : x_n \mapsto x_{n-1} \mapsto \cdots \mapsto x_1 \mapsto x_n$.}
\label{fig_sigma-i}
\end{figure}
\end{center}

\begin{ex}
Proposition~\ref{gen not OP} 
together with Corollary \ref{cor_multDelta^2} implies the links of Figure \ref{fig_non-bo_link} have complements whose fundamental groups are not bi-orderable.  
Those complements are homeomorphic to each other.
\end{ex}

\begin{center}
\begin{figure}
\includegraphics[width=4in]{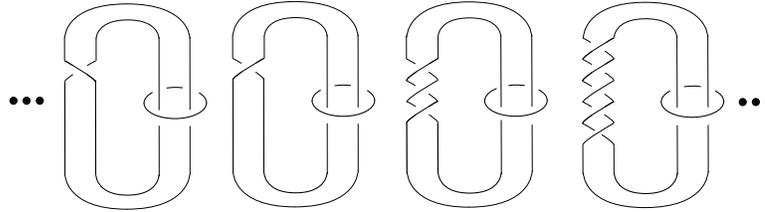}
\caption{Links whose groups are not bi-orderable.}
\label{fig_non-bo_link}
\end{figure}
\end{center}

\subsection{Pure braids} 
\label{subsection_Pure-braids}

As is well-known, there is a homomorphism $B_n \to \mathcal{S}_n$ of the $n$-strand braid group onto the permutation group of $n$ letters, and the kernel is the group of pure braids $P_n$.

\begin{prop} \label{pure OP}
Every pure braid $\b \in P_n$ is order-preserving, and in fact preserves every standard bi-order of $F_n$.
\end{prop}


\begin{proof}
For a pure braid $\beta$, the image of the Artin representatation is pure symmetric. 
This completes the proof by Proposition~\ref{prop_pure-symmetric}. 
\end{proof}


The $2$nd power $\sigma_i^2$ of each generator $\sigma_i$ is a pure braid. 
Thus we have the following.

\begin{cor}
The braids $\s_i^2 \in P_n$ are order-preserving.
\end{cor}

\begin{ex}
The 2-strand braid $\s_1^2$ gives rise to the examples of links in Figure \ref{fig_bo_link} whose complements have bi-orderable fundamental groups.  
All the link complements are homeomorphic to one another. 
It is clear that they are homeomorphic with 
$D_2 \times S^1$, whose fundamental group is isomorphic with $F_2 \times \Z$.
\end{ex}

\begin{center}
\begin{figure}
\includegraphics[width=4in]{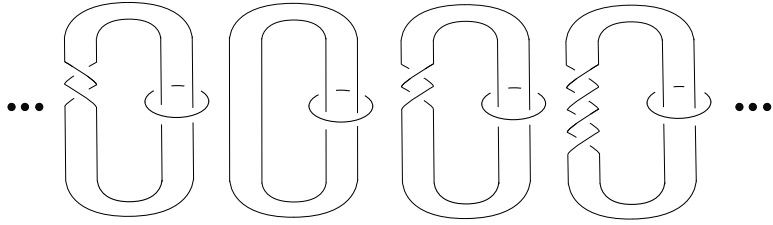}
\caption{Links whose groups are bi-orderable.}
\label{fig_bo_link}
\end{figure}
\end{center}

 By Corollary~\ref{cor_power}, we immediately obtain the following.

\begin{cor}
For every braid $\b$ some power $\b^k$ is order-preserving.  
The fundamental group $\pi_1(\T_{\b^k})$ is bi-orderable and 
may be regarded as a normal subgroup of index $k$ in $\pi_1(\T_{\b})$.
\end{cor}

\subsection{Periodic braids} 

Suppose that  $\b \in B_n$ is a periodic braid. 
It is known (see for example \cite[page~30]{Gonzalez11}) that 
there exists an integer $k \in {\Bbb Z}$ such that 
$\b$ is conjugate to either 

\begin{enumerate}
\item[Type 1.] 
$ (\delta_n \sigma_1)^k = (\sigma_1 \sigma_2 \cdots \sigma_{n-1} \sigma_1)^k$ or 

\item[Type 2.] 
$\delta_n^k = (\sigma_1 \sigma_2 \cdots \sigma_{n-1})^k $. 
\end{enumerate}

\begin{thm}
\label{thm_periodicBraid}
Let $\b \in B_n$ be a periodic braid. 
\begin{enumerate}
\item[(1)]
If $\b$ is conjugate to $ (\delta_n \sigma_1)^k $ for some $k$, then $\b$ is order-preserving and 
$\pi_1({\Bbb T}_{\beta})$ is bi-orderable. 

\item[(2)] 
If $\beta$ is conjugate to  $\delta_n^k$ for some $k$ 
with $k \not\equiv 0$ $\pmod n$, 
then $\b$ is not order-preserving and 
$\pi_1({\Bbb T}_{\beta})$ can not be bi-orderable. 
\end{enumerate}
\end{thm}

\noindent
Note that $\delta_n^k$ is order-preserving when $k \equiv 0$ $\pmod n$ 
since $\delta_n^n= \Delta^2$.  
Theorem~\ref{thm_periodicBraid} is a consequence of  \cite[Theorem~1.5]{BRW05}: 
If $\b$ is a periodic braid, then 
${\Bbb T}_\b$ is a Seifert fibered $3$-manifold. 
If $\b$ is of type 1, then ${\Bbb T}_\b$ has no exceptional fibres. 
If $\b$ is of type $2$ and if $k$ is not a multiple of $n$, then ${\Bbb T}_\b$ has an exceptional fibre. 

We will give an alternative, more explicit proof of Theorem~\ref{thm_periodicBraid} in Section~\ref{subsection_Alternative}.

\subsection{Disk twists}
\label{subsection_DiskTwists}
We review a method to construct links in $S^3$ 
whose complements are homeomorphic to each other. 
Let $L$ be a link in $S^3$. 
We denote a  tubular neighborhood  of $L$ by $\mathcal{N}(L)$, 
and  the exterior of $L$, that is $S^3 \setminus \mathrm{int}(\mathcal{N}(L))$ by $\mathcal{E}(L)$. 
Suppose that $L$ contains an unknot  $K$ as a sublink. 
Then $\mathcal{E}(K)$ (resp. $\partial \mathcal{E}(K)$)
is homeomorphic to a solid torus (resp. torus). 
We denote the link $ L\setminus K$ by $L_K$. 
Taking a disk $D$ bounded by the longitude of $ \mathcal{N}(K)$, 
we define two homeomorphisms 
$$T_D:  \mathcal{E}(K) \rightarrow   \mathcal{E}(K)$$
and 
$$\mathfrak{t}_D:   \mathcal{E}(L) 
(=  \mathcal{E}(K \cup L_K)) \rightarrow \mathcal{E}(K \cup T_D(L_K))$$ 
as follows. 
We cut $ \mathcal{E}(K)$ along $D$. 
 We have  resulting two sides obtained from $D$. 
Then we reglue the two sides by rotating either of the sides  $360$ degrees 
so that the mapping class of the restriction 
$T_D|_{\partial  \mathcal{E}(K)}: \partial \mathcal{E}(K)\rightarrow 
\partial \mathcal{E}(K)$ defines the right-handed Dehn twist about $\partial D$, 
see Figure~\ref{fig_DiskTwist}(1). 
Such an operation defines the homeomorphism 
$T_D:  \mathcal{E}(K) \rightarrow   \mathcal{E}(K)$. 
If $m$ segments of $L_K$ pass through $D$, then 
$T_D(L_K)$ is obtained from $L_K$ by adding a full twist braid $\Delta_m^2$ near $D$. 
In the case  $m=2$, see Figure~\ref{fig_DiskTwist}(2). 
For example, if $L$ is equivalent to $\mathrm{br}(\beta)$ for some $\beta \in B_n$ 
and if $K$ is taken to be the braid axis $A$ of $\beta$, 
then $T_D(L_K)$ is equivalent to the closure of $\beta \Delta_n^2$. 
Notice that 
$T_D:  \mathcal{E}(K) \rightarrow   \mathcal{E}(K)$ determines the latter  homeomorphism 
$$\mathfrak{t}_D:  \mathcal{E}(L) 
(=  \mathcal{E}(K \cup L_K)) \rightarrow \mathcal{E}(K \cup T_D(L_K)).$$
We call $\mathfrak{t}_D$  the {\it (right-handed) disk twist about} $D$.

For any integer $m \ne 0$, we have a homeomorphism of 
the $m$th power $T_D^{m}:  \mathcal{E}(K) \rightarrow   \mathcal{E}(K)$ 
so that 
$T_D^{m}|_{\partial  \mathcal{E}(K)}: \partial  \mathcal{E}(K) \rightarrow \partial  \mathcal{E}(K)$
is the $m$th power of the right-handed Dehn twist about $\partial D$. 
Observe that  $T_D^{m}$  converts $L = K \cup L_K$ into a link 
 $K \cup T_D^{m}(L_K)$ in $S^3$ 
 such that $S^3 \setminus L$ is homeomorphic to $S^3 \setminus (K \cup T^{m}(L_K))$. 
 We denote by $\mathfrak{t}_D^{m}$, a homeomorphism: $\mathcal{E}(L)(= \mathcal{E}(K \cup L_K)) \rightarrow \mathcal{E}(K \cup T_D^{m}(L_K))$ 
 and call  $\mathfrak{t}_D^{m}$  the $m$th power of {\it (right-handed) disk twist} $\mathfrak{t}_D$ about $D$.

\begin{center}
\begin{figure}
\includegraphics[width=5in]{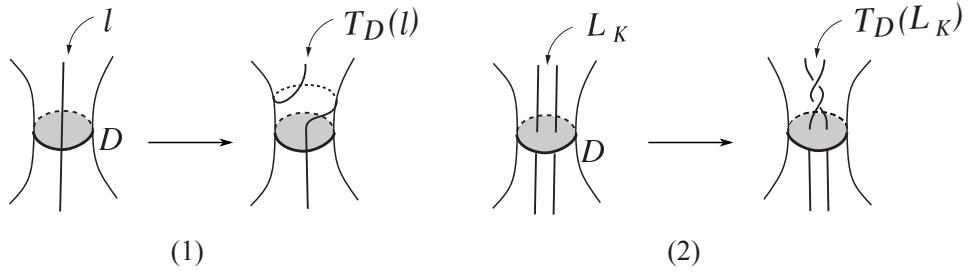}
\caption{(1) Image of $\ell$ under $T_D$, 
where $\ell$ is an arc on $\partial \mathcal{E}(K)$ 
which passes through $\partial D$. 
(2) Local picture of $L_K$ and its image $T_D(L_K)$. }
\label{fig_DiskTwist}
\end{figure}
\end{center}

\subsection{Alternative proof of  Theorem~\ref{thm_periodicBraid}}
\label{subsection_Alternative}

\begin{proof}[Proof of Theorem~\ref{thm_periodicBraid}] 
We prove the claim~(1).  Consider the pure $2$-strand braid $\s_1^2$.  
Then $\mathrm{br}(\s_1^2)$ is a link consisting of three unknotted components, including the axis.  Performing a disk twist $n$ times on one of the components of the closed braid $\widehat{\s_1^2}$
converts the braided link of the 2-strand braid $\sigma_1^2$ to the braided link of the $(n+2)$-strand braid $\beta'$, which is conjugate to  the type 1 braid 
$\sigma_1 \sigma_2 \cdots \sigma_{n+1} \sigma_1$, 
see Figure~\ref{fig_ex_disktwist}.  
But the disk twist being a homeomorphism of the complement of the link, we see that $\T_{\s_1^2} \simeq \T_{\beta'}$.  
But since $\s_1^2$ is pure, $\T_{\s_1^2}$ has bi-orderable fundamental group, in fact isomorphic with $F_2 \times \Z.$  Hence the fundamental group of $\T_{\beta'}$ is bi-orderable and 
$\sigma_1 \sigma_2 \cdots \sigma_{n+1} \sigma_1$ is order-preserving. 
Thus the $k$th power $(\sigma_1 \sigma_2 \cdots \sigma_{n+1} \sigma_1)^k$ is also order-preserving.

We turn to the claim~(2). 
We use the basepoint $p_1$ in the interior of $D_n$. 
There exists a homeomorphism $\phi: D_n \rightarrow D_n$ 
which represents $\delta_n \in \mathrm{Mod}(D_n)$ 
such that the induced map 
$\phi_*: \pi_1(D_n, p_1) \rightarrow  \pi_1(D_n, p_1)$ 
gives rise to the following automorphism\footnote{The product $\sigma_1 \sigma_2 \cdots \sigma_{n-1} \in \mathrm{Aut}(\pi_1(D_n, p_1))$ of $\sigma_i$'s in (\ref{equation_sigma}) 
is given by $x_1 \mapsto x_n$, $x_j \to x_n x_{j-1} x_n^{-1}$ if $j \ne 1$. This is equal to 
$\delta_n \in \mathrm{Aut(\pi_1(D_n, p_1))}$ in (\ref{equation_delta}) 
up to an inner automorphism.}  
on $F_n= \langle x_1, \cdots, x_n \rangle$: 
\begin{eqnarray}
\label{equation_delta}
 \delta_n: x_1 \mapsto x_n \mapsto x_{n-1} \mapsto \cdots \mapsto x_2 \mapsto x_1,
\end{eqnarray}
see Figure~\ref{fig_sigma-i}(2). 
The orbit of $x_1$ is nontrivial  (since $x_1 \ne x_n$) and finite. 
By Proposition~\ref{finite orbit}, $\delta_n$ is not order-preserving. 
Then $k \not\equiv 0$ $\pmod n$ if and only if $\delta_n^k$ has a nontrivial finite orbit of $x_1$. 
Thus $\delta_n^k$ is not order-preserving. 
\end{proof}

\begin{prop}
The half twist $\Delta_n \in B_n$ is order-preserving if and only if $n$ is odd. 
\end{prop}

\begin{proof}
If $n=2m+1$, 
then $\Delta_{n}$ is conjugate to $(\sigma_1 \sigma_2 \cdots \sigma_{2m}\sigma_1)^m$, 
and if $n=2m$, 
then $\Delta_{n}$ is conjugate to $(\sigma_1 \sigma_2 \cdots \sigma_{2m-1})^m$ with $m \not\equiv 0$ $\pmod {2m}$. 
By Theorem~\ref{thm_periodicBraid}, we finish the proof. 
\end{proof}

\begin{center}
\begin{figure}
\includegraphics[width=3in]{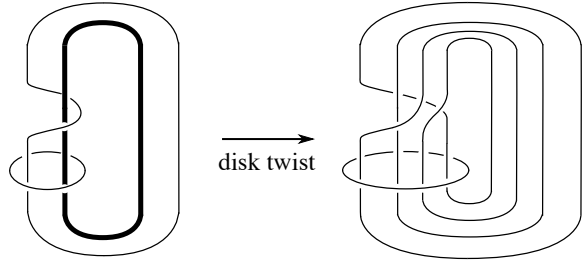}
\caption{$n$th power of the disk twist 
converts the braided link of $\s_1^2 $ to that of $\sigma_1 \sigma_2 \cdots \sigma_{n+1}\sigma_1$. 
($n=2$ in this case.)}
\label{fig_ex_disktwist}
\end{figure}
\end{center}

As a special case of Theorem \ref{thm_periodicBraid}(2), we see that the periodic 3-strand braid $\s_1\s_2$ is not order-preserving.
Another way to see this is to observe that the 3-strand braid $\sigma_1\sigma_2$ gives the following automorphism on $F_3 = \langle x,y,z \rangle$ 
by using the Artin representation $B_3 \rightarrow \mathrm{Aut}(F_3)$ 
$$z \mapsto y \mapsto x \mapsto xyzy^{-1}x^{-1}$$
One can show this doesn't preserve a bi-order of $F_3$ as follows.  
Assume some bi-order $<$ were preserved by the map, then supposing without loss of generality 
that $x<y$, we would have $xyzy^{-1}x^{-1} < x$ (applying $\s_1\s_2$), hence $yzy^{-1} < x$ (conjugation invariance), so $xyx^{-1} < xyzy^{-1}x^{-1}$ (applying $\s_1\s_2$) and then $xyx^{-1} < x$ (transitivity) and finally the contradiction $y < x$ (conjugation again).

On the other hand 
$\sigma_1 \sigma_2^{-1} \in B_3$ is pseudo-Anosov, and in fact it is 
the simplest pseudo-Anosov $3$-strand braid; 
see Section~\ref{subsection_small-dil}. 
We thank George Bergman for pointing out the following argument.  It will also follow from more general result 
in Section~\ref{subsection_Non-order-preserving}; see 
Corollary \ref{cor_beta-1m}.

\begin{thm}
The braid $\s_1\s_2^{-1} \in B_3$ is not order-preserving. 
\end{thm}

\begin{proof}
Let $x, y, z$ be the free generators of $F_3$.  
Using the Artin representation $B_3 \rightarrow \mathrm{Aut}(F_3)$, one sees that 
the action of 
$\s_1\s_2^{-1}$ is
$$x \mapsto xzx^{-1}, \quad y \mapsto x, \quad z \mapsto z^{-1}yz.$$
Consider the orbit of the element $w = y^{-1}x$ under this action.  Assuming there is a bi-order $<$ of $F_3$ invariant under this action, 
we may assume without loss of generality  that $ 1<w$, and therefore all elements of the orbit of $w$ are positive.
Moreover $1<w$ implies $y<x$ and since $w \mapsto zx^{-1}$ we have $ y<x<z$.

Now the calculation $$w \mapsto zx^{-1} \mapsto z^{-1}yzxz^{-1}x^{-1} = (z^{-1}y)zx^{-1}x(xz^{-1})x^{-1}$$ and the facts that $z^{-1}y$ and $x(xz^{-1})x^{-1}$ are negative show that the action is decreasing on the orbit of $zx^{-1}$, which is also the orbit of $w$.

Calculating the preimages of the generators, we have the action of $\s_1\s_2^{-1}$ expressed as
$$y \mapsto x, \quad y^{-1}xyzy^{-1}x^{-1}y \mapsto y, \quad y^{-1}xy \mapsto z,$$
and therefore
$$y^{-1}xyz^{-1}y^{-1}x^{-1}yy \mapsto y^{-1}x = w.$$
But notice that $y^{-1}xyz^{-1}y^{-1}x^{-1}yy = w(yz^{-1})y^{-1}(x^{-1}y)y$ and since the expressions in the parentheses are $<1$ 
(i.e, $yz^{-1}<1$ and $y^{-1}(x^{-1}y)y<1$), 
we conclude that the action is increasing on the orbit of $w$.  This contradiction shows that an invariant bi-order of $F_3$ cannot exist.
\end{proof}

\subsection{Explicit orderings preserved by periodic braids of type 1}
\label{explicit}
We have seen that for $n \ge 3$ the root 
$\delta_n \sigma_1 $ of the full twist $\Delta^2 \in B_n$ preserves an ordering of $F_n$.  In this section we will explicitly construct uncountably many such orderings.  

Recall that $\delta_n \in B_n$  induces the following automorphism on $F_n$,  
using the basepoint in the interior of $D_n$ 
$$\delta_n: x_1 \mapsto x_n \mapsto x_{n-1} \mapsto \cdots \mapsto x_2 \mapsto x_1,$$
see (\ref{equation_delta}). 
By (\ref{equation_sigma}) 
$\sigma_1 \in B_n$ induces the following automorphism on $F_n$, 
using the same basepoint: 
$$\sigma_1: x_1 \mapsto x_2,\ x_2 \mapsto x_2 x_1 x_2^{-1},\ x_j \mapsto x_j\ \mbox{if}\ j  \ne 1,2.$$
Thus 
the automorphism $\delta_n \sigma_1$ on $F_n $  is given by
$$\delta_n \sigma_1: 
x_1 \to x_n, \quad x_2 \to x_2, \quad x_3 \to x_2 x_1 x_2^{-1}, \quad x_4 \to x_3, \quad \cdots , \quad x_n \to x_{n-1}.$$
(For the case $n=3$, just take the first three terms above.)

Here is another way to realize this automorphism of $F_n$, 
elaborating on Example 3.6 in \cite{LRR08}.  
Fix $n \ge 3$ and consider the free group $F_2 = \langle u, v \rangle$ of rank $2$ 
and the homomorphism of $F_2$ onto the cyclic group $G = \langle t \mid t^{n-1} = 1 \rangle$ given by 
$u \to t$ and $v \to 1$.  The kernel $\mathcal{K} = \mathcal{K}_n$ of this map is a normal subgroup of $F_2$, of index $n-1$.  If we realize $F_2$ as the fundamental group of a bouquet of two circles labelled $u$ and $v$, 
the covering space corresponding to $\mathcal{K}$ is a finite planar graph as pictured in Figure~\ref{fig_graph_covering} whose fundamental group is free of rank $n$.  
Using the basepoint in the covering space depicted in Figure~\ref{fig_graph_covering}, we see that $\mathcal{K}$ has the free generators $z_1, \dots, z_n$, where 
\begin{eqnarray*}
z_1 = v, \quad z_2 &=& u^{n-1}, \quad z_3 = u^{n-2}vu^{2-n}, \quad z_4 = u^{n-3}vu^{3-n}, 
\\
\quad \dots \quad , z_{n-1} &=& u^2vu^{-2}, \quad z_n = uvu^{-1}.
\end{eqnarray*}

Now consider the automorphism of the normal subgroup $\mathcal{K}$ of $F_2$ 
given by $\phi(x) = uxu^{-1}$.  This is not an inner automorphism of $\mathcal{K}$ but it is an inner automorphism of the larger group $F_2$.  
Therefore, if we take {\it any} bi-ordering of $F_2$, then its restriction to 
$\mathcal{K}$ will be preserved by $\phi$.  By inspection, the action of $\phi$ on $\mathcal{K}$ is given by
$$z_1 \to z_n, \quad z_2 \to z_2, \quad z_3 \to z_2z_1z_2^{-1}, \quad z_4 \to z_3, \quad \cdots , \quad z_n \to z_{n-1}.$$
Under the isomorphism $\mathcal{K} \cong F_n$ given by $z_i \mapsto x_i$ for each $i$, we see that $\phi$ corresponds to 
$\delta_n \sigma_1$.  
One can use this isomorphism to see that any bi-ordering of $F_2$ restricted to $\mathcal{K}$ provides an ordering of $F_n$ invariant under $\delta_n \sigma_1$.  
 Finally note that any ordering of $F_n$ respected by $\delta_n \sigma_1$ will also be respected by $(\delta_n \sigma_1)^k$ for any integer $k$.
 
 \begin{center}
\begin{figure}
\includegraphics[width=4.5in]{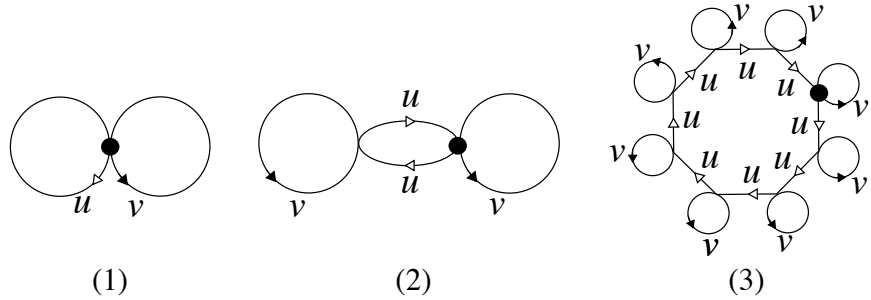}
\caption{(1) A bouquet of two circles $u$ and $v$. 
(2) Covering space corresponding to $\mathcal{K}_n$ when $n= 3$. 
(3)  Covering space corresponding to $\mathcal{K}_n$ when $n= 9$.}
\label{fig_graph_covering}
\end{figure}
\end{center}

Note that we can use, for example, any standard ordering of $F_2$.  
However, this ordering restricted to $\mathcal{K}$ cannot be a standard ordering (defined using the lower central series of $\mathcal{K}$) by Proposition \ref{sym not id}.

\subsection{Tensor product of braids}

Given braids $\a \in B_m$ and $\b \in B_n$ one can form the $(m+n)$-strand braid $\a \otimes \b \in B_{m+n}$ with $\a$ on the first $m$ strings and $\b$ on the last $n$ strings, 
but no crossings between any of the first $m$ strings with any of the last $n$ strings (Figure~\ref{fig_tensorProduct});
see for example \cite[page~69]{KT08}.   The action of 
$\a \otimes \b$ on $F_{m+n} \cong F_m \star F_n$ is just the free product $\a \star \b : F_m \star F_n \to F_m \star F_n$.

 \begin{center}
\begin{figure}
\includegraphics[width=3in]{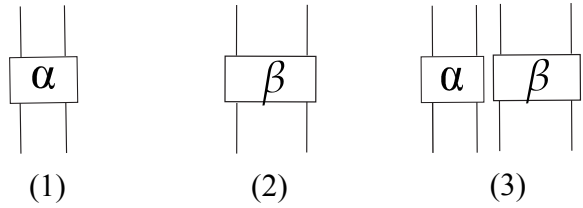}
\caption{(1) $\alpha \in B_m$. 
(2)  $\beta \in B_n$. 
(3) $\alpha \otimes \beta \in B_{m+n}$.}
\label{fig_tensorProduct}
\end{figure}
\end{center}

The following Lemma and Proposition are proved in \cite[Corollary 4]{Rol16}. 
\begin{lem}
\label{lem_tensor}
Suppose $(G, <_G)$ and $(H, <_H)$ are bi-ordered groups.  Then there is a bi-ordering of $G \star H$ which extends the orderings of the factors and such that whenever $\phi : G \to G$ and $\psi : H \to H$ are order-preserving automorphisms, the ordering of $G \star H$ is preserved by the automorphism $\phi \star \psi : G \star H \to G \star H$.
\end{lem}

\begin{prop}
The braids $\a \in B_m$ and $\b \in B_n$ are order-preserving if and only if  $\a \otimes \b \in B_{m+n}$ is order-preserving.
\end{prop}


There is a natural inclusion $B_m \subset B_{m+n}$ ($n \ge 1$) given by $\b \mapsto \b \otimes 1_n$, where $1_n$ is the identity braid of $B_n$. 

\begin{cor}\label{extension}
A braid $\b \in B_m$ is order-preserving if and only if $\b \otimes 1_n \in B_{m+n}$ is order-preserving.
\end{cor}

\subsection{Do order-preserving braids form a subgroup?}

Let $OP_n \subset B_n$ denote the set of $n$-strand braids which are order-preserving.  
It is clear that $\b \in OP_n$ if and only if $\b^{-1} \in OP_n$
and that the identity braid belongs to $OP_n$.  Moreover, $P_n \subset OP_n$ by Proposition \ref{pure OP}.  It is natural to ask whether $OP_n$ forms a subgroup of $B_n$, in other words whether $OP_n$ is closed under multiplication.  For $n=2$ the answer is affirmative.  Noting that $\Delta_2^2 = \s_1^2$ we conclude from Propositions \ref{gen not OP} and \ref{pure OP} and Corollary \ref{cor_multDelta^2} that $OP_2$ consists of exactly the 2-strand braids $\s_1^k$ with $k$ even.  Therefore $OP_2$ is exactly the subgroup $P_2$.

%

\begin{prop}
For $n > 2$ the set $OP_n$ is not a subgroup of $B_n$.
\end{prop}

\begin{proof}
Consider the $n$-strand braids $\a = \s_1\s_2\s_1$ and $\b = \s_1^{-2}$.  Then $\a$ is order-preserving, being an extension in $B_n$ of type 1 periodic braid $\s_1\s_2\s_1 \in B_3$ (see Corollary \ref{extension}), and $\b$ is order-preserving, being a pure braid.  But 
$\a\b = \s_1\s_2\s_1^{-1}$ is not order-preserving, as it is conjugate to $\s_2$ which is not order-preserving 
by Proposition~\ref{gen not OP}. 
\end{proof}


Although not a subgroup for $n>2$, $OP_n$ is a large subset of $B_n$: 

\begin{prop}  
\label{prop_OP_n-generate}
For $n > 2$ the set $OP_n$ of order-preserving $n$-braids generates $B_n$.
\end{prop}

\begin{proof}
We saw above that $\s_1\s_2\s_1^{-1}$ is a product of braids $\a, \b \in OP_n$; therefore $\s_2$ is also a product of appropriate conjugates of $\a$ and $\b$; these conjugates are also in $OP_n$.  But all the generators $\s_i$ of $B_n$ are conjugate to each other, and therefore are also products of elements of $OP_n$.
It follows that all braids are products of elements of $OP_n$.
\end{proof}

\section{Small volume cusped hyperbolic 3-manifolds}
\label{section_hyperbolic}
It is known by Gabai-Meyerhoff-Milley~\cite{GMM09} that 
the Weeks manifold is the unique closed orientable hyperbolic $3$-manifold of smallest volume. 
Its fundamental group is not left-orderable; see Calegari-Dunfield~\cite{CD03}. 
In this section we will see that certain minimum volume orientable $n$-cusped 3-manifolds can be distinguished by orderability properties of their fundamental groups. 
We also prove that some orientable hyperbolic $n$-cusped $3$-manifolds with the smallest known volumes 
have  bi-orderable fundamental groups.  

Let $C_3$ and $C_4$ be the chain links with $3$ and $4$ components as in Figure~\ref{fig_nChainLink}(1) and (2).  
For $n \ge 5$, let $C_n$ be {\it the minimally twisted} {\it $n$-chain link}; see \cite[Section~1]{KPR12} for the 
definition of such a link. 


Figure~\ref{fig_nChainLink}(3) and (4) show $C_5$ and $C_6$.   
Let ${\Bbb W}_n$ be the $n$-fold cyclic cover  over one component of the Whitehead link complement ${\Bbb W}$. 
(See Figure~\ref{fig_2Whitehead}(2) for ${\Bbb W}_3$.) 
It seems that $S^3 \setminus C_n$ and ${\Bbb W}_{n-1}$ play an important role 
for the study of the minimal volume hyperbolic $3$-manifolds with $n$ cusps, where $n \ge 3$,  
as we shall see below.

\subsection{One cusp}

Cao-Meyerhoff~\cite{CM01} proved that 
a minimal volume orientable hyperbolic $3$-manifold with $1$ cusp is homeomorphic to either 
the complement of figure-eight knot in $S^3$ or its sibling manifold (which can be described as 5/1 Dehn surgery on one component of the Whitehead link). 
We note that the sibling manifold cannot be described as a knot complement, as its first homology group is 
$\Z \oplus \Z/5\Z$.  

Like the figure-eight knot complement, it can be described as a punctured torus bundle over $S^1$.

\begin{thm}\label{thm: one cusp}
The figure-eight knot complement has bi-orderable fundamental group.  The fundamental group of its sibling is not bi-orderable.
\end{thm}

\begin{proof}
The first assertion is proved in \cite{PR03}, using the monodromy 
$\left(\begin{array}{cc}2 & 1 \\1 & 1\end{array}\right)$
associated with the fibration, which has two positive eigenvalues, and Theorem \ref{Perron-Rolfsen}.
The sibling has the monodromy $\left(\begin{array}{cc}-2 & -1 \\-1 & -1\end{array}\right)$, see \cite[Proposition~3 and Note]{HMW90}.  This has the two negative eigenvalues $(-3 \pm \sqrt{5})/2$.
By Theorem~\ref{Clay-Rolfsen}, 
the sibling has non-bi-orderable fundamental group.
\end{proof}

\subsection{Two cusps}

Agol~\cite{Agol10} proved that
the minimal volume orientable hyperbolic $3$-manifold with $2$ cusps 
is homeomorphic to either the Whitehead link complement ${\Bbb W}$ or the $(-2,3,8)$-pretzel link complement ${\Bbb W}'$.

\begin{thm}
The fundamental group of ${\Bbb W}$ is bi-orderable. 
The fundamental group of ${\Bbb W}'$ is not bi-orderable. 
\end{thm}

\begin{proof}
The first assertion follows from Theorem~\ref{thm_Whitehead}. 
We shall prove in Theorem~\ref{thm_BigEntropy} that 
$\delta_5 \sigma_1^2 \in B_5$ is not order-preserving. 
This together with Proposition~\ref{prop_OPiffBi-ord} 
implies the second assertion, since $\mathrm{br}(\delta_5 \sigma_1^2)$ is equivalent to the 
$(-2,3,8)$-pretzel link; see Figure~\ref{fig_123412pretzel}. 
\end{proof}

\begin{center}
\begin{figure}
\includegraphics[width=4.8in]{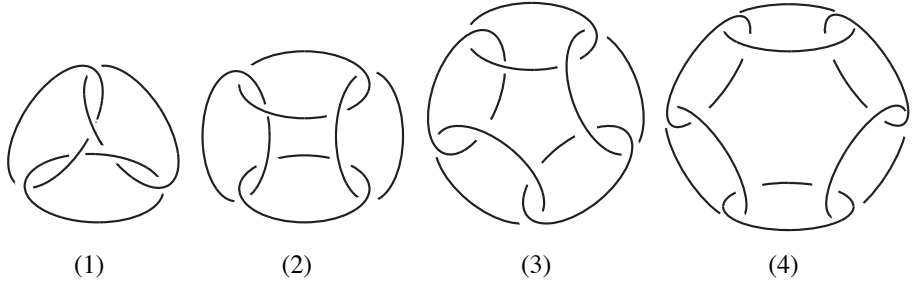}
\caption{(1) $C_3$. (2) $C_4$. (3) $C_5$. (4) $C_6$.}
\label{fig_nChainLink}
\end{figure}
\end{center}

\subsection{Four cusps} 
It is proved by 
Ken'ichi Yoshida~\cite{Yoshida13} that the minimal volume orientable hyperbolic $3$-manifold with $4$ cusps is homeomorphic to 
$S^3 \setminus C_4$.

\begin{thm}
\label{thm_C4}
The complement 
$S^3 \setminus C_4$ has bi-orderable fundamental group. 
\end{thm}

\begin{proof}
Consider  $\sigma_1^{-2}  \sigma_2^{2} \in P_3$. 
Then $\pi_1(S^3 \setminus \mathrm{br}(\sigma_1^{-2}  \sigma_2^{2})))$ is bi-orderable 
by Proposition~\ref{pure OP}. 
We now see that 
$S^3 \setminus \mathrm{br}(\sigma_1^{-2}  \sigma_2^{2})$ is homeomorphic to $S^3 \setminus C_4$. 
Take a disk $D$ bounded by the thickened unknot $K \subset \mathrm{br}( \sigma_1^{-2} \sigma_2^2)$, 
 see Figure~\ref{fig_4Chain}(1). 
 Then the left-handed disk twist 
 $\mathfrak{t}_D^{-1}$ sends  $\mathcal{E}(\mathrm{br}( \sigma_1^{-2} \sigma_2^2))$ 
 to the exterior of a link $L$ as shown in Figure~\ref{fig_4Chain}(2). 
 We observe that  $L$ is equivalent to $C_4$, see Figure~\ref{fig_4Chain}(2)(3).  
 Thus $\pi_1(S^3 \setminus C_4) $ is bi-orderable. 
 \end{proof}

 \begin{rem}
 \label{rem_disktwistC4}
 In the proof of Theorem~\ref{thm_C4}, 
 if we replace $\mathfrak{t}_D^{-}$ with 
 $\mathfrak{t}_D^{+}$, then  $\mathfrak{t}_D^{+}$ sends  $\mathcal{E}(\mathrm{br}( \sigma_1^{-2} \sigma_2^2))$ to 
 $\mathcal{E}(\mathrm{br}( \sigma_1^{-2} \sigma_3 \sigma_2 \sigma_3))$. 
 It follows that  $S^3 \setminus \mathrm{br}( \sigma_1^{-2} \sigma_2^2)$ is homeomorphic to $S^3 \setminus \mathrm{br}( \sigma_1^{-2} \sigma_3 \sigma_2 \sigma_3)$. 
 By Proposition~\ref{prop_OPiffBi-ord}, 
 $\sigma_1^{-2} \sigma_3 \sigma_2 \sigma_3 (= \sigma_3 \sigma_1^{-2} \sigma_2 \sigma_3 )$ 
 which is conjugate to $ \sigma_1^{-2} \sigma_2 \sigma_3^2 \in B_4 $ 
 is order-preserving. 
 \end{rem}

 \subsection{Five cusps}

It has been conjectured \cite{Agol10} that
$S^3 \setminus C_5$ has the smallest volume  
among  orientable hyperbolic $3$-manifolds with $5$ cusps.

 \begin{thm}
\label{thm_C5}
The complement 
$S^3 \setminus C_5$ has bi-orderable fundamental group. 
\end{thm}

 \begin{proof}  
 We take $\sigma_1^{-2} \sigma_2^{-2} \sigma_3^{-2} \in P_4$. 
 Then $\pi_1(S^3 \setminus \mathrm{br} (\sigma_1^{-2} \sigma_2^{-2} \sigma_3^{-2}))$ is bi-orderable. 
To see $S^3 \setminus \mathrm{br}(\sigma_1^{-2}  \sigma_2^{-2} \sigma_3^{-2})$ is 
 homeomorphic to $S^3 \setminus C_5$, 
 we first take a disk $D$ bounded by the thickened unknot 
 $K \subset \mathrm{br}(\sigma_1^{-2}  \sigma_2^{-2} \sigma_3^{-2})$ as in Figure~\ref{fig_5Chain}(1). 
 The disk twist $\mathfrak{t}_D$ sends $\mathcal{E}(\mathrm{br}(\sigma_1^{-2}  \sigma_2^{-2} \sigma_3^{-2}))$ 
 to the exterior of a link $L$ as shown in Figure~\ref{fig_5Chain}(2). 
 Next we take a disk $D'$ bounded by the thickened unknot $K' \subset L$, see Figure~\ref{fig_5Chain}(2). 
 Then $\mathfrak{t}_{D'}$ sends $\mathcal{E}(L)$ to the exterior of a link $L'$ as in  Figure~\ref{fig_5Chain}(2). 
 We see that $L'$ is equivalent to $C_5$. 
 Thus $S^3 \setminus \mathrm{br}(\sigma_1^{-2}  \sigma_2^{-2} \sigma_3^{-2}) \simeq S^3 \setminus C_5$, and  
  $\pi_1(S^3 \setminus C_5)$ is bi-orderable.
\end{proof}

\begin{center}
\begin{figure}
\includegraphics[width=4.2in]{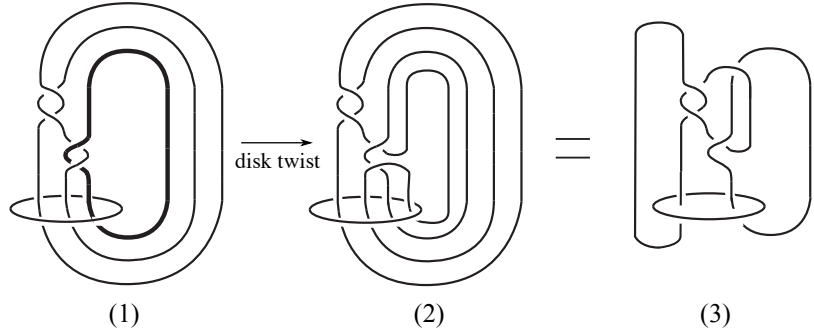}
\caption{$S^3 \setminus \mathrm{br}(\sigma_1^{-2}  \sigma_2^{2})$ is homeomorphic to $S^3 \setminus C_4$. 
(1) $\mathrm{br}(\sigma_1^{-2}  \sigma_2^{2})$. (2)(3) Links which are equivalent to $C_4$.}
\label{fig_4Chain}
\end{figure}
\end{center}

\begin{center}
\begin{figure}
\includegraphics[width=5in]{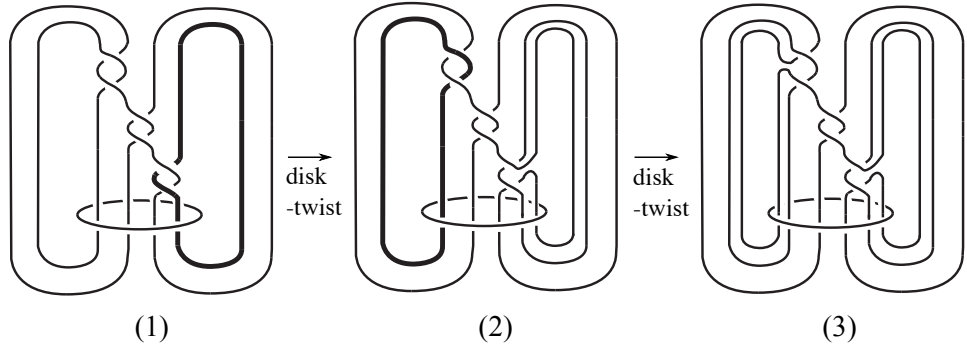}
\caption{$S^3 \setminus \mathrm{br}(\sigma_1^{-2}  \sigma_2^{-2} \sigma_3^{-2})$ is homeomorphic to $S^3 \setminus C_5$. 
(1) $ \mathrm{br}(\sigma_1^{-2}  \sigma_2^{-2} \sigma_3^{-2})$. 
(3) Link which is equivalent to $C_5$.}
\label{fig_5Chain}
\end{figure}
\end{center}

\subsection{Three cusps} 
 
The $3$-chain link complement $N= S^3 \setminus C_3$, which Gordon and Wu \cite{GW99} named the {\it magic manifold}, has the smallest known volume  
among orientable hyperbolic $3$-manifolds with $3$ cusps. 
The magic manifold $N$ is homeomorphic to $ S^3 \setminus \mathrm{br}(\sigma_1^2 \sigma_2^{-1})$, 
where $\sigma_1^2 \sigma_2^{-1} \in B_3$. 
To see this, we take a disk $D$ bounded by the thickened unknot $K \subset \mathrm{br}(\sigma_1^2 \sigma_2^{-1})$ as shown in Figure~\ref{fig_3Chain}(1). 
Then $\mathfrak{t}^{-1}_D$ sends $\mathcal{E}(\mathrm{br}(\sigma_1^2 \sigma_2^{-1}))$ 
to the exterior of a link $L$ as shown in Figure~\ref{fig_3Chain}(2). 
Since $L$ is equivalent to $C_3$, it follows that $ S^3 \setminus \mathrm{br}(\sigma_1^2 \sigma_2^{-1})$ is homeomorphic to $N$. 
We note that 
the $2$nd power $(\sigma_1^2 \sigma_2^{-1})^2  $ is a pure braid. 
Hence $\pi_1(S^3 \setminus \mathrm{br}((\sigma_1^2 \sigma_2^{-1})^2))$ is bi-orderable, 
and it has index $2$ in $\pi_1(N)$.

\begin{ques}
\label{ques_magic}
Is  $\pi_1(N)$ bi-orderable? 
In other words is $\sigma_1^2 \sigma_2^{-1} \in B_3$ order-preserving? 
\end{ques}

\begin{center}
\begin{figure}
\includegraphics[width=3in]{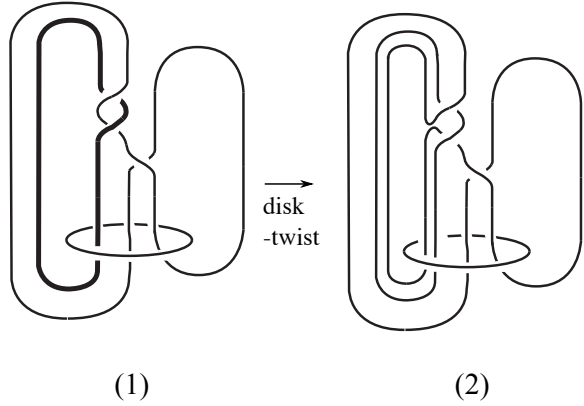}
\caption{$S^3 \setminus \mathrm{br}(\sigma_1^2 \sigma_2^{-1})$ is homeomorphic to $S^3 \setminus C_3$. 
(1) $\mathrm{br}(\sigma_1^2 \sigma_2^{-1})$. (2) Link which is equivalent to $C_3$.} 
\label{fig_3Chain}
\end{figure}
\end{center}

\subsection{$n$ cusps for $n \ge 6$} 

It is a conjecture by Agol~\cite{Agol10} that  for  $n \le 10$, 
the minimally twisted $n$ chain link complement $S^3 \setminus C_n$ has the minimal volume among orientable hyperbolic $3$-manifold with $n$ cusps. 

The following theorem follows from Lemmas~\ref{lem_L1} and \ref{prop_C6L1} below. 

\begin{thm}
\label{thm_C6}
The complement 
$S^3 \setminus C_6$ has bi-orderable fundamental group. 
\end{thm}

Let $L_1$ be the $6$-circle link as shown in Figure~\ref{fig_11223344}(2).

\begin{center}
\begin{figure}
\includegraphics[width=3.5in]{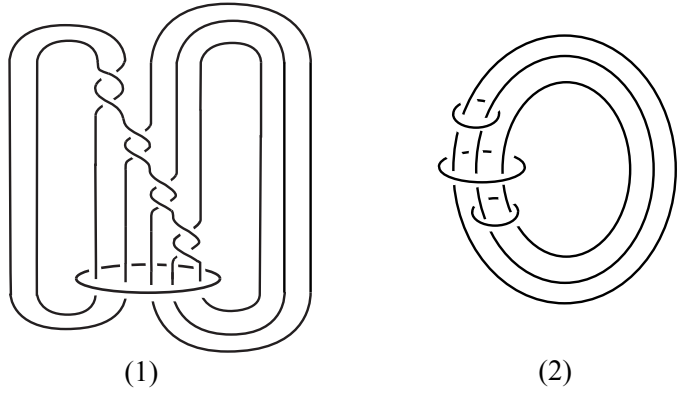}
\caption{(1) $\mathrm{br}(\sigma_1^{-2} \sigma_2^{-2} \sigma_3^{-2} \sigma_4^{-2})$. 
(2) The $6$-circled link $L_1$. 
The complements of these links are homeomorphic to each other.} 
\label{fig_11223344}
\end{figure}
\end{center}

\begin{lem}
\label{lem_L1}
Let $\beta =  \sigma_1^{-2} \sigma_2^{-2} \sigma_3^{-2} \sigma_4^{-2} \in P_5$. 
Then $S^3 \setminus L_1$ is homeomorphic to 
$S^3 \setminus \mathrm{br}(\beta)$, and 
$\pi_1(S^3 \setminus L_1)$ is bi-orderable. 
\end{lem}
 
\begin{proof}
We consider the two trivial components $K$, $K' \subset  \mathrm{br}(\beta)$
bounded by the closure of the first string and the closure of the last string. 
Let $D$ and $D'$ be the disk bounded by $K$ and $K'$. 
We do the disk twists $\mathfrak{t}_D$ and $\mathfrak{t}_{D'}$. 
Then the resulting link is equivalent to $L_1$. 
Thus $S^3 \setminus  \mathrm{br}(\beta)$ is homeomorphic to $S^3 \setminus L_1$. 
Since $\beta$ is a pure braid, this implies that $\pi_1(S^3 \setminus L_1)$ is bi-orderable. 
\end{proof}

We thank Ken'ichi Yoshida who conveys the following argument to the authors. 

\begin{lem}
\label{prop_C6L1}
$S^3 \setminus L_1$  is homeomorphic to $S^3 \setminus C_6$. 
\end{lem}

\begin{proof}
We take four $3$-punctured spheres embedded in the exterior $\mathcal{E}(L_1)$ 
as shown in Figure~\ref{fig_yosida}(2). 
The four $3$-punctured spheres are also embedded in $\mathcal{E}(C_6)$ 
as shown in Figure~\ref{fig_6Chain}(2). 
Let $U$ (resp. $U'$) be a subset of $\mathcal{E}(L_1)$ (resp. a subset of $\mathcal{E}(C_6)$) bounded by these  $3$-punctured spheres. 
Note that $\mathcal{E}(L_1)$ (resp. $\mathcal{E}(C_6)$) are the double of $U$ (resp. $U'$) with respect to the four 
$3$-punctured spheres, see Figure~\ref{fig_yosida} and Figure~\ref{fig_6Chain}(3). 
We now see that $U$ is homeomorphic to $U'$. 
This is  enough to prove the lemma since the double of a manifold is uniquely determined. 

We push the bottom shaded $3$-punctured sphere in $U$ as shown in Figure~\ref{fig_yosida}(3), 
and  deform it into the shaded $3$-punctured sphere as in Figure~\ref{fig_yosida}(6). 
(The middle red colored  annulus in (3) is modified into the bottom red colored  annulus in (6).)
As a result, we get $U'$. 
\end{proof}

\begin{center}
\begin{figure}
\includegraphics[width=4.5in]{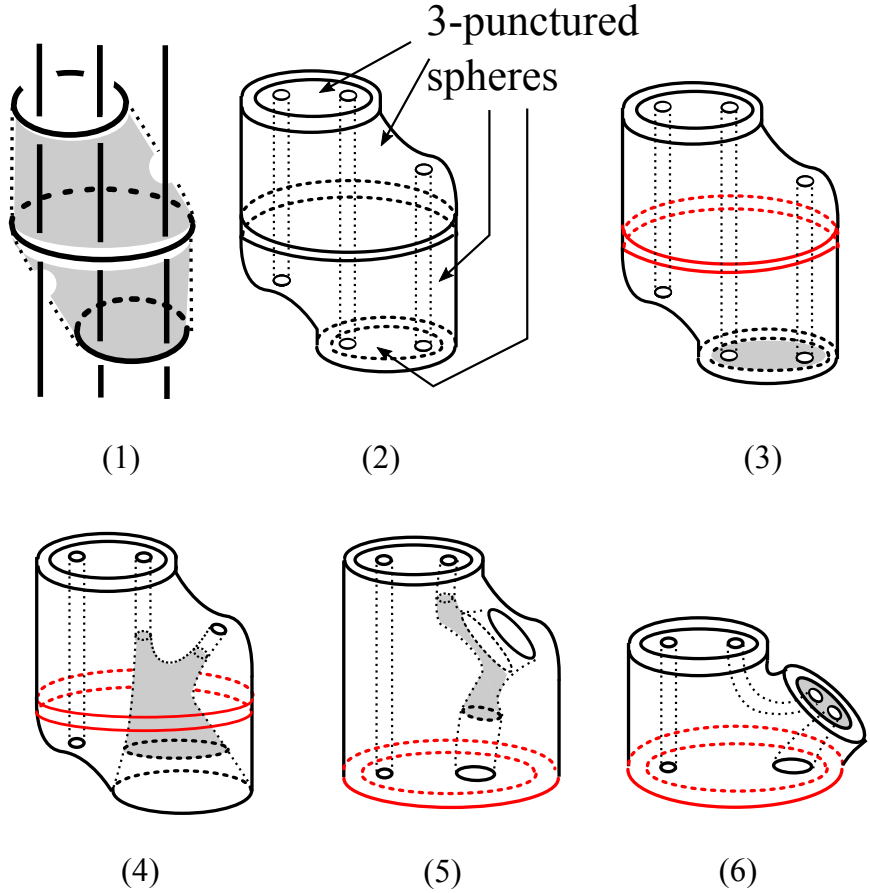}
\caption{(1) (A part of) $L_1$. 
(2) A subset $U \subset \mathcal{E}(L_1)$ bounded by four $3$-punctured spheres. 
Figures~(3)--(6) show a modification of $U$ into $U'$. 
(cf. Figure~\ref{fig_6Chain}(2).)} 
\label{fig_yosida}
\end{figure}
\end{center}

\begin{center}
\begin{figure}
\includegraphics[width=4in]{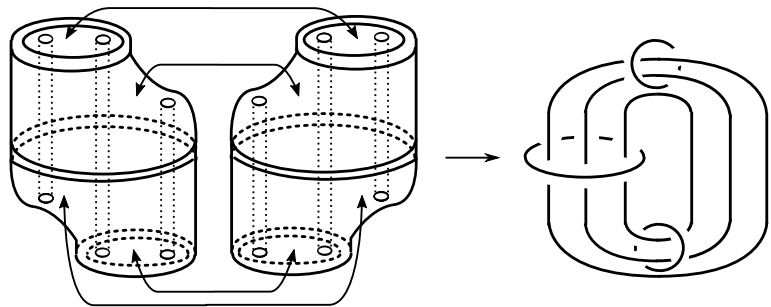}
\caption{The double of $U$  with respect to the four 
$3$-punctured spheres is homeomorphic to $\mathcal{E}(L_1)$.} 
\label{fig_double}
\end{figure}
\end{center}

\begin{center}
\begin{figure}
\includegraphics[width=4.5in]{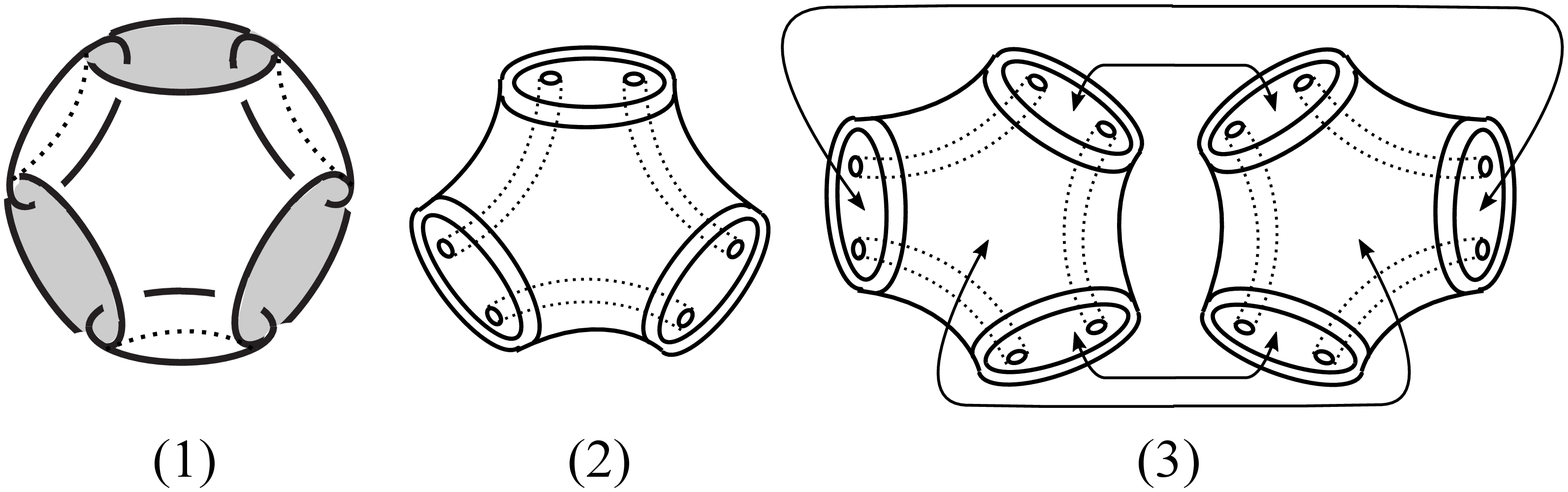}
\caption{(1) $C_6$. 
(2) A subset $U'\subset \mathcal{E}(C_6)$ bounded by four $3$-punctured spheres. 
(cf. Figure~\ref{fig_yosida}(6).) 
(3) The double of $U'$  with respect to the four 
$3$-punctured spheres is homeomorphic to $\mathcal{E}(C_6)$.} 
\label{fig_6Chain}
\end{figure}
\end{center}


Kaiser-Purcell-Rollins proved that the volume of ${\Bbb W}_{n-1}$ is strictly smaller than that of $S^3 \setminus C_n$ 
if $12 \le n \le 25$ or $n \ge 60$; see \cite[Theorems~1.1, 4.1]{KPR12}. 
At the time of this writing, it seems that 
${\Bbb W}_{n-1}$ has the smallest known volume among orientable hyperbolic $3$-manifolds with $n$ cusps if $n \ge 11$. 
(See also Table~1 in \cite{KPR12}.)

\begin{prop}
\label{prop_WhiteheadCover}
The fundamental group of  ${\Bbb W}_n$  is bi-orderable for each $n \ge 2$. 
\end{prop} 

\begin{proof}
Let 
$p: {\Bbb W}_n \rightarrow {\Bbb W}$ be an $n$-fold cyclic covering map. 
Then  $p_*(\pi_1 ({\Bbb W}_n))$ has index $n$ in $\pi_1({\Bbb W})$. 
Since  $\pi_1 ({\Bbb W})$ is bi-orderable, so does $\pi_1 ({\Bbb W}_n)$. 
\end{proof}

\subsection{A result by M.~Baker and a question}

We learned the following theorem from Hidetoshi Masai. 

\begin{thm}[Baker~\cite{Baker02}]
\label{thm_Baker}
The $6$-circle link $L_1$ is an arithmetic link. 
Every link $L$ in $S^3$ occurs as a sublink of a link $J$ such that 
$S^3 \setminus J$ is a covering space of $S^3 \setminus L_1$. 
In particular $L$ is a sublink of an arithmetic link in $S^3$. 
\end{thm}

Theorem~\ref{thm_Baker} and Lemma~\ref{lem_L1} immediately give us the following.

\begin{thm}
\label{thm_addingLink} 
Let $L$ be a link in $S^3$. 
Then $L$ is a sublink of a link  whose complement has bi-orderable fundamental group. 
\end{thm}

The link with $4$ components shown in Figure~\ref{fig_2Whitehead}(2) 
is obtained from $C_3$ adding a trivial knot. 
Its complement  which is homeomorphic to ${\Bbb W}_3$ has bi-orderable fundamental group (Proposition~\ref{prop_WhiteheadCover}), 
although we have not decided yet $N= S^3 \setminus C_3$ has  bi-orderable fundamental group. 
We ask the following.

\begin{ques}
Let $M$ be a $3$-manifold. 
Does there exist a knot $K$ in $M$ such that 
$\pi_1(M \setminus K)$ is bi-orderable?
\end{ques}

\section{Non-order-preserving braids} 
\label{section_nonOP}

In this section, 
we give some sequences  of non-order-preserving  braids. 
Our sequences include examples of pseudo-Anosov braids with small dilatations. 
We also provide examples of non-order-preserving pseudo-Anosov braids with arbitrary large dilatations.

\subsection{Pseudo-Anosov braids with smallest dilatations} 
\label{subsection_small-dil}

Let $\beta \in B_n$ be a pseudo-Anosov braid. 
Let $\lambda(\beta)>1$ be the dilatation (i.e,  stretch factor) of the corresponding pseudo-Anosov mapping class 
$\beta \in \mathrm{Mod}(D_n)$. 
(See \cite{FM12} for the definition of dilatations, for example.)
It is known that there exists a minimum, denoted by $\delta(D_n)$ among  dilatations for all pseudo-Anosov $n$-strand braids. 
The explicit values of minimal dilatations $\delta(D_n)$'s  are determined for $3 \le n \le 8$ 
by Ko-Los-Song, Ham-Song and Lanneau-Thiffeault. 
The following $n$-strand braids realize $\delta(D_n)$. 
 
\begin{itemize}
\item $n=3$; 
$ \sigma_1^{-1} \sigma_2 \in B_3$, see \cite{Matsuoka86, Handel97} for example.

\item $n=4$; 
$ \sigma_1^{-1} \sigma_2 \sigma_3\in B_4$, see \cite{KLS02}. 
(See also \cite{HS07}.) 

\item $n=5$; 
$\delta_5 \sigma_4^{-1} \sigma_3^{-1} \in B_5$, see \cite{HS07}. 

\item $n=6$; 
$\sigma_2 \sigma_1 \sigma_2 \sigma_1 \delta_5^2 \in B_6$, see \cite{LT11}. 
 
 \item $n=7$; 
 $\sigma_4^{-2} \delta_7^2 \sim \delta_7^2 \sigma_6^{-1} \sigma_5^{-1} \in B_7$, see  \cite{LT11}. 
  
 \item $n=8$; 
 $ \sigma_2^{-1} \sigma_1^{-1} \delta_8^5 \sim 
 \delta_8^5 \sigma_7^{-1} \sigma_6^{-1} 
 \in B_8$, see  \cite{LT11}. 
\end{itemize}
Here $\b \sim \b'$ means that $\b$ is conjugate to $\b'$. 
We will prove in Section~\ref{subsection_Non-order-preserving} that 
for $3 \le n \le 8$ except $n=6$, 
the above $n$-strand braids with the smallest dilatations are not order-preserving 
(Corollary~\ref{cor_beta-1m}, Theorem~\ref{thm_small-dil}, Lemma~\ref{lem_mini8braid}). 
At the time of this writing we do not know whether the above  $6$-strand braid 
$\sigma_2 \sigma_1 \sigma_2 \sigma_1 \delta_5^2$ is order-preserving or not. 
We remark that 
Question~\ref{ques_magic} is equivalent to the one asking whether $ \sigma_2 \sigma_1 \sigma_2 \sigma_1 \delta_5^2$  is order-preserving or not, 
since $S^3 \setminus \mathrm{br}((\sigma_2 \sigma_1 \sigma_2 \sigma_1 \delta_5^2)^{-1} \Delta^2)$ is homeomorphic to the magic manifold $N$, 
see \cite[Page 39 and Corollary~3.2]{KT11}.

\subsection{Sequences of pseudo-Anosov braids} 
\label{subsection_Non-order-preserving}

\begin{thm} 
\label{thm_BigEntropy}
For $n \ge 3$ and $k \ge 1$, 
$\delta_n \sigma_1^{2k} = (\sigma_1 \sigma_2 \cdots \sigma_{n-1}) \sigma_1^{2k} \in B_n$ 
is not order-preserving. 
\end{thm}

If $n \ge  5$, then 
$\delta_n \sigma_1 \sigma_2 \in B_n$ is pseudo-Anosov, 
see \cite[Theorem~3.11]{HK06}. 
This claim can be proved by using the criterion by Bestvina-Handel~\cite{BH05}. 
(See also \cite[Section~2.4]{HK06}.) 
Since $\delta_n \sigma_1^2 \sim \delta_n \sigma_1 \sigma_2$, 
the braid $\delta_n \sigma_1^2$ 
is  pseudo-Anosov. 
By using the same criterion, it is not hard to see that 
$\sigma_1^{2k-2} \delta_n \sigma_1 \sigma_2 \sim \delta_n \sigma_1^{2k}$ is pseudo-Anosov for $n \ge 5$ and $k \ge 1$. 
Viewing the transition matrix associated to the pseudo-Anosov braid $\delta_n \sigma_1^{2k}$, 
we see that the largest eigenvalue of the transition matrix, 
which is equal to  $\lambda(\delta_n \sigma_1^{2k})$,  goes to $\infty$ as $n$ goes to $\infty$.

\begin{proof}[Proof of Theorem~\ref{thm_BigEntropy}]
We fix a basepoint of $\pi_1(D_n)$ in the interior of $D_n$. 
By (\ref{equation_delta}) 
we have the automorphism 
$$\delta_n: x_1 \mapsto x_n \mapsto x_{n-1} \mapsto \cdots \mapsto x_3 \mapsto x_2 \mapsto x_1.$$
We turn to braid   $\sigma_1^{2k} \in B_n$.  
By (\ref{equation_sigma})  
we have the following automorphism induced by $\sigma_1 \in B_n$ 
using the same basepoint of $\pi_1(D_n)$:  
$$\sigma_1: x_1 \mapsto x_2, \quad x_2 \mapsto x_2 x_1 x_2^{-1}, \quad x_j \mapsto x_j\ \mbox{for}\ j \ne 1,2.$$  
Computing the automorphism $\sigma_1^2$ on $F_n $, we get 
$$\sigma_1^2: x_1 \mapsto x_2 x_1 x_2^{-1}, \quad x_2 \mapsto x_2 x_1 x_2 x_1^{-1} x_2^{-1}, \quad x_j \mapsto x_j\ \mbox{for}\ j \ne 1,2.$$
In the same manner, the automorphism $\sigma_1^{2k}$ for $k \ge 1$ is given by 
\begin{eqnarray*}
\sigma_1^{2k}: x_1 &\mapsto& (x_2 x_1)^{k-1} x_2 x_1 x_2^{-1} (x_1^{-1} x_2^{-1})^{k-1}, \quad 
x_2 \mapsto (x_2 x_1)^k x_2 (x_1^{-1} x_2^{-1})^k, 
\\
x_j &\mapsto& x_j \quad \mbox{for}\ j \ne 1,2.
\end{eqnarray*}
Thus  $\delta_n \sigma_1^{2k}$ 
induces the following automorphism on $F_n$. 
\begin{eqnarray*}
\phi := \delta_n \sigma_1^{2k}: \ x_1 \mapsto x_n \mapsto x_{n-1} \mapsto \cdots \mapsto x_3 &\mapsto& (x_2 x_1)^k x_2 (x_1^{-1} x_2^{-1})^k, 
\\
x_2 &\mapsto&  (x_2 x_1)^{k-1} x_2 x_1 x_2^{-1} (x_1^{-1} x_2^{-1})^{k-1}. 
\end{eqnarray*}
We assume that  $\phi $ preserves some bi-ordering $<$ on $F_n$. 
Without loss of generality, 
we may assume that $x_1 < x_2$. 
Then we have 
$$(x_2 x_1)^{k-1} x_2 x_1 x_2^{-1} (x_1^{-1} x_2^{-1})^{k-1} = (x_2 x_1)^{k}  x_1 (x_1^{-1} x_2^{-1})^{k} 
<  (x_2 x_1)^k x_2 (x_1^{-1} x_2^{-1})^k$$ 
by conjugation invariance of the bi-ordering. 
This implies that 
$$x_2 = \phi^{-1} ((x_2 x_1)^{k-1} x_2 x_1 x_2^{-1} (x_1^{-1} x_2^{-1})^{k-1}) 
< \phi^{-1} ((x_2 x_1)^k x_2 (x_1^{-1} x_2^{-1})^k) = x_3,$$
since $\phi^{-1}$ also preserves the same bi-ordering $<$. 
Hence we have 
$x_1 < x_2 < x_3$. 
Notice that $x_3 $ is contained in the orbit $\mathcal{O}(x_1)$ under $\phi$, 
i.e, 
$$x_1 \mapsto \phi(x_1)= x_n \mapsto \cdots \mapsto x_3 \mapsto (x_2 x_1)^k x_2 (x_1^{-1} x_2^{-1})^k \mapsto \cdots$$ 
Since $x_1 < x_3$, this implies that 
the ordering $<$ is increasing on $\mathcal{O}(x_1)$, 
i.e, 
$$x_1 < \phi(x_1)= x_n < \phi^2(x_1) = x_{n-1} <  \cdots < 
x_3 < (x_2 x_1)^k x_2 (x_1^{-1} x_2^{-1})^k < \cdots$$
Notice that $x_1 < x_2$ implies that 
$$\phi(x_1)= x_n < \phi(x_2) =  (x_2 x_1)^{k-1} x_2 x_1 x_2^{-1} (x_1^{-1} x_2^{-1})^{k-1}.$$
This together with $x_1 < x_n$ tells us that 
\begin{equation} 
\label{equation_BigEntropy1}
x_1  < (x_2 x_1)^{k-1} x_2 x_1 x_2^{-1} (x_1^{-1} x_2^{-1})^{k-1}.
\end{equation}
We multiply the both sides of (\ref{equation_BigEntropy1}) by $(x_2 x_1)^{k-1} x_2$ on the right. 
$$x_1(x_2 x_1)^{k-1} x_2  < (x_2 x_1)^{k-1} x_2 x_1 x_2^{-1} (x_1^{-1} x_2^{-1})^{k-1} (x_2 x_1)^{k-1} x_2 = (x_2 x_1)^{k-1} x_2 x_1.$$
Thus $$(x_1 x_2)^k= x_1(x_2 x_1)^{k-1} x_2 < (x_2 x_1)^{k-1} x_2 x_1 = (x_2 x_1)^k.$$ 
On the other hand, since $x_2 < x_3$ and $x_3 < (x_2 x_1)^k x_2 (x_1^{-1} x_2^{-1})^k$, we have 
\begin{equation}
\label{equation_BigEntropy2} 
x_2<  (x_2 x_1)^k x_2 (x_1^{-1} x_2^{-1})^k.
\end{equation}
Multiply the both sides of (\ref{equation_BigEntropy2}) by $(x_2 x_1)^k$ on the right. 
$$x_2 (x_2 x_1)^k<  (x_2 x_1)^k x_2 (x_1^{-1} x_2^{-1})^k (x_2 x_1)^k = x_2(x_1 x_2)^k.$$ 
This implies that 
$ (x_2 x_1)^k < (x_1 x_2)^k (< (x_2 x_1)^k)$, a  contradiction. 
Thus $\delta_n \sigma_1^{2k}$ 
is not order-preserving. 
\end{proof}

We have $\mathrm{br}(\delta_n \sigma_1^{2k}) = \mathrm{br}(\sigma_1^{2k} \delta_n)$, 
and it is equivalent to the $(-2, 2k+1, 2n-2)$-pretzel link;
see Figure~\ref{fig_123412pretzel} in the case $n= 5$ and $k=1$.  
By Theorem~\ref{thm_BigEntropy} and Propsition~\ref{prop_OPiffBi-ord}, 
we have the following. 

\begin{cor} 
For each $n \ge 3$ and $k \ge 1$, 
the fundamental group of the $(-2, 2k+1, 2n-2)$-pretzel link complement is not bi-orderable. 
\end{cor}

We turn to another sequence of braids. 

\begin{thm}
\label{thm_BigEntropy2}
Let $n \ge 3$ and $k \ge 1$. 
Then $\delta_n^{n-1} \sigma_1^{2k} = (\sigma_1 \sigma_2 \cdots \sigma_{n-1})^{n-1} \sigma_1^{2k} \in B_n$ 
is not order-preserving. 
\end{thm}

\begin{proof}
Fix a basepoint of $\pi_1(D_n)$ in the interior of $D_n$. 
By using (\ref{equation_delta}), 
we calculate 
the automorphism $\delta_n^{n-1}= (\sigma_1 \sigma_2 \cdots \sigma_{n-1})^{n-1}$. 
$$\delta_n^{n-1}: x_1 \mapsto x_2 \mapsto \cdots \mapsto x_{n-1} \mapsto x_n \mapsto x_1.$$ 
In the same manner as in the proof of Theorem~\ref{thm_BigEntropy}, 
we have the following automorphism on $F_n$ 
obtained from  the braid $\delta_n^{n-1} \sigma_1^{2k}$ 
using the same basepoint of $\pi_1(D_n)$: 
\begin{eqnarray*}
\delta_n^{n-1} \sigma_1^{2k}: 
 x_2 \mapsto x_3 \mapsto \cdots \mapsto  x_{n-1} \mapsto  x_n &\mapsto& (x_2 x_1)^{k-1} x_2 x_1 x_2^{-1} (x_1^{-1} x_2^{-1})^{k-1},
\\
x_2 &\mapsto&  (x_2 x_1)^k x_2 (x_1^{-1} x_2^{-1})^k. 
\end{eqnarray*}
Assume that $\delta_n^{n-1} \sigma_1^{2k}$ preserves some bi-ordering on $F_n$. 
In the same manner as in the proof of Theorem~\ref{thm_BigEntropy}, 
one arrives a contradiction. 
\end{proof}

\begin{center}
\begin{figure}
\includegraphics[width=4.8in]{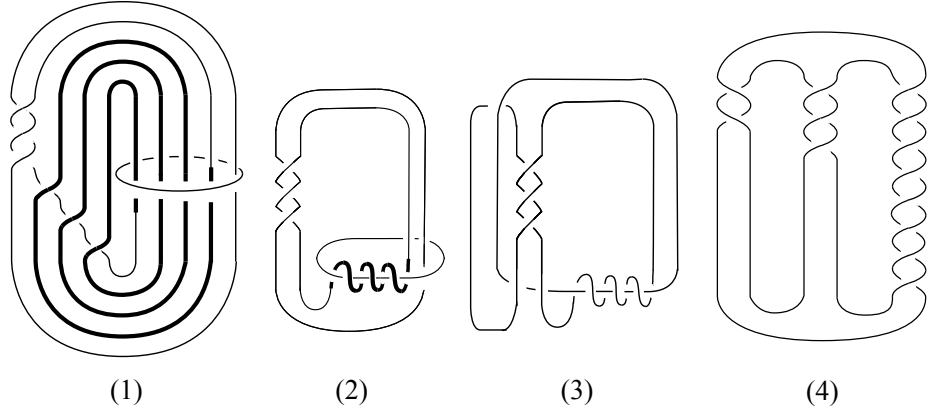}
\caption{
$\mathrm{br}(\sigma_1^{2k}\delta_n)$  
is equivalent to the $(-2, 2k+1, 2n-2)$-pretzel link ($n=5$, $k=1$ in this case), 
see (1) $\rightarrow$ (2) $\rightarrow$ (3) $\rightarrow$ (4). 
(1) $ \mathrm{br} (\sigma_1^2 \delta_5 )$. (4)  $(-2,3,8)$-pretzel link.}
\label{fig_123412pretzel}
\end{figure}
\end{center}

\begin{cor}
\label{cor_beta-1m}
Let $n \ge 3$ and $k \ge 1$. 
Then $\sigma_1^{-2k} \delta_n= \sigma_1^{-2k+1} \sigma_2 \sigma_3 \cdots \sigma_{n-1} \in B_n$ is not order-preserving. 
In particular $\sigma_1^{-1} \sigma_2 \sigma_3  \cdots \sigma_{n-1} \in B_n$ is not  order-preserving for each $n \ge 3$. 
\end{cor}

\begin{proof}
By Theorem~\ref{thm_BigEntropy2}, 
the inverse 
$(\delta_n^{n-1} \sigma_1^{2k})^{-1}$ is not order-preserving. 
Hence 
$(\delta_n^{n-1} \sigma_1^{2k})^{-1} \Delta^{2\ell}$ is not order-preserving 
for each $\ell \in {\Bbb Z}$. 
We have 
\begin{eqnarray*}
(\delta_n^{n-1} \sigma_1^{2k})^{-1} \Delta^2 &=& 
((\sigma_1 \sigma_2 \cdots \sigma_{n-1})^{n-1} \sigma_1^{2k})^{-1} (\sigma_1 \sigma_2 \cdots \sigma_{n-1})^n
\\
&=&  \sigma_1^{-2k} (\sigma_{n-1}^{-1} \sigma_{n-2}^{-1} \cdots \sigma_1^{-1})^{n-1} (\sigma_1 \sigma_2 \cdots \sigma_{n-1})^n
\\
&=&   \sigma_1^{-2k}  \sigma_1 \sigma_2 \cdots \sigma_{n-1} 
(= \sigma_1^{-2k} \delta_n). 
\end{eqnarray*}
This completes the proof.
\end{proof}

\begin{rem}
For each $n \ge 3$, 
$\sigma_1^{-2} \delta_n= 
\sigma_1^{-1} \sigma_2 \sigma_3 \cdots \sigma_{n-1} \in B_n$ is pseudo-Anosov, 
see \cite[Theorem~3.9]{HK06}. 
It is not hard to see that 
$$\sigma_1^{-2k} \delta_n= \sigma_1^{-2k+1} \sigma_2 \cdots \sigma_{n-1}$$ 
is pseudo-Anosov  for $n \ge 3$ and $k \ge 1$ 
by using the Bestvina-Handel criterion.  
The Nielsen-Thurston types of a mapping class and its inverse are the same. 
Thus $\delta_n^{n-1} \sigma_1^{2k} \in B_n$ in Theorem~\ref{thm_BigEntropy2} is pseudo-Anosov for $n \ge 3$ and $k \ge 1$. 
\end{rem}

Using similar modifications of links as in Figure~\ref{fig_123412pretzel}, 
we verify that  $\mathrm{br}(\sigma_1^{-2k} \delta_n)$ is equivalent to 
the $(-2, -2k+1, 2n-2)$-pretzel link. 
Thus we have the following. 

\begin{cor}
For each $n \ge 3$ and $k \ge 1$, 
the fundamental group of the $(-2, -2k+1, 2n-2)$-pretzel link complement is not bi-orderable. 
\end{cor}

We turn to the last sequence of braids which plays an important role for the study of pseudo-Anosov minimal dilatations, see \cite{HK06,KT11}. 

\begin{thm}
\label{thm_small-dil}
For each $n \ge 5$, 
$$\delta_n^2 \sigma_{n-1}^{-1} \sigma_{n-2}^{-1} = (\sigma_1 \sigma_2 \cdots \sigma_{n-1})  (\sigma_1 \sigma_2 \cdots \sigma_{n-3}) \in B_n$$ 
is not order-preserving. 
\end{thm}

The braid 
 $\delta_n^2 \sigma_{n-1}^{-1} \sigma_{n-2}^{-1}$ is reducible if $n$ is even. 
 (In fact it is easy to find an essential simple closed curve on $D_n$ containing the punctures 
 labelled $2,4, \cdots, n-2$ which is invariant under the corresponding mapping class of $\mathrm{Mod}(D_n)$.) 
 If $n$ is odd, the braid $\delta_n^2 \sigma_{n-1}^{-1} \sigma_{n-2}^{-1}$ is pseudo-Anosov, 
 see  \cite[Theorem~3.11, Figure~18(c)]{HK06}.

\begin{proof}[Proof of Theorem~\ref{thm_small-dil}]
We fix a basepoint of $\pi_1(D_n)$ in the interior of $D_n$. 
From the automorphism $\sigma_i$ of $F_n$ for each $i= 1, \cdots, n-1$ (see (\ref{equation_sigma})), 
we calculate the automorphism $\sigma_i^{-1}$ on $F_n$: 
$$\sigma_i^{-1}: x_i \mapsto x_i^{-1} x_{i+1} x_i, \quad x_{i+1} \mapsto x_i, \quad x_j \mapsto x_j\ \mbox{if}\ j \ne i, i+1.$$
By (\ref{equation_delta}) together with automorphisms 
$\sigma_i^{-1}$ for $i= n-1,n-2$, 
we see that 
the following automorphism is arisen from 
the braid $\delta_n \sigma_{n-1}^{-1} \sigma_{n-2}^{-1} = \sigma_1 \sigma_2 \cdots \sigma_{n-3}$. 
\begin{eqnarray*}
\delta_n \sigma_{n-1}^{-1} \sigma_{n-2}^{-1} &:& 
x_1 \mapsto x_{n-2} \mapsto x_{n-3} \mapsto \cdots \mapsto x_2 \mapsto x_1, 
\\
&\ &x_{n-1} \mapsto x_{n-2}^{-1} x_{n-1} x_{n-2}, \quad 
x_{n} \mapsto x_{n-2}^{-1} x_{n} x_{n-2}, 
\end{eqnarray*}
see Figure~\ref{fig_sigma_n-2}. 
Thus the automorphism 
$\psi:=  \delta_n (\delta_n \sigma_{n-1}^{-1} \sigma_{n-2}^{-1}): F_n \rightarrow F_n$ is 
given by 
\begin{eqnarray*}
\psi= \delta_n^2  \sigma_{n-1}^{-1} \sigma_{n-2}^{-1}  &:&  
x_1 \mapsto x_{n-2}^{-1} x_{n} x_{n-2}, 
\\
&\ & x_2 \mapsto x_{n-2}, \ x_3 \mapsto x_1, \cdots, x_{n-2} \mapsto x_{n-4}, \  x_{n-1} \mapsto x_{n-3}, 
\\
&\ & x_n \mapsto x_{n-2}^{-1} x_{n-1} x_{n-2}.
\end{eqnarray*}
Suppose that $n$ is even. 
Then the orbit of $x_2$ under $\psi$ is nontrivial and finite: 
$$x_2 \mapsto x_{n-2} \mapsto x_{n-4} \mapsto \cdots \mapsto x_4 \mapsto x_2.$$
Thus $\psi$ does not preserve any left-ordering on $F_n$ by Proposition~\ref{finite orbit}.

Suppose that $n$ is odd. 
Then the above automorphism $\psi: F_n \rightarrow F_n$ is described as follows. 
\begin{eqnarray*}
x_{n-1} \mapsto x_{n-3} \mapsto \cdots \mapsto x_2 \mapsto x_{n-2} \mapsto x_{n-4} \mapsto x_1 &\mapsto&  x_{n-2}^{-1} x_{n} x_{n-2}, 
\\
x_n &\mapsto&  x_{n-2}^{-1} x_{n-1} x_{n-2}. 
\end{eqnarray*}
We suppose that $\psi$ preserves a bi-ordering $<$ of $F_n$. 
Without of loss of generality, we may assume that $x_{n-1} < x_n$. 
By the conjugation invariance of the bi-ordering, 
$x_{n-2}^{-1} x_{n-1} x_{n-2} < x_{n-2}^{-1} x_{n} x_{n-2}$ holds. 
Since $\psi^{-1}$ preserves the same ordering $<$, 
this implies that $x_n  < x_1$, 
see the above definition of $\psi$. 
Hence we have $x_{n-1}< x_n < x_1$. 
This tells us that the ordering $<$ is increasing on the orbit $\mathcal{O}(x_{n-1})$, i.e, 
$$x_{n-1} < x_{n-3} < \cdots < x_2 < x_{n-2} < x_{n-4} < x_1 <  x_{n-2}^{-1} x_{n} x_{n-2}< \cdots$$
In particular $x_{n-2}< x_{n-2}^{-1} x_{n} x_{n-2}$ which implies that $x_{n-2}< x_n$ by the conjugation invariance of the bi-ordering. 
This together with the $\psi$-invariance of $<$ gives us that 
$x_{n-4}< x_{n-2}^{-1} x_{n-1} x_{n-2}$. 

On the other hand, by $x_{n-1}< x_{n-2}$, we have 
$x_{n-2}^{-1} x_{n-1} x_{n-2} <  x_{n-2}$ by the conjugation invariance again.  
However we have 
$x_{n-2}< x_{n-4}<x_{n-2}^{-1} x_{n-1} x_{n-2}  $, in particular 
$$x_{n-2}< x_{n-2}^{-1} x_{n-1} x_{n-2} \ (< x_{n-2}), $$
which is a contradiction. 
This  completes the proof. 
\end{proof}

\begin{center}
\begin{figure}
\includegraphics[width=5in]{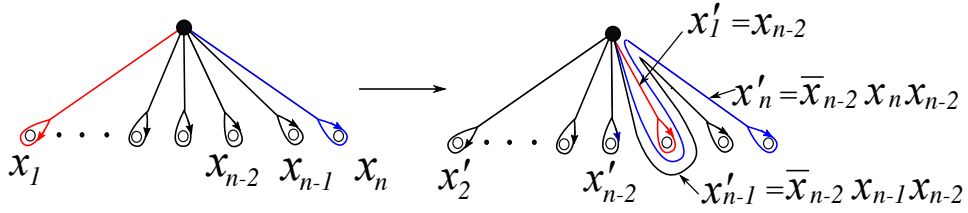}
\caption{
$\delta_n \sigma_{n-1}^{-1} \sigma_{n-2}^{-1} =\sigma_1 \sigma_2 \cdots \sigma_{n-3}: F_n \rightarrow F_n$.}
\label{fig_sigma_n-2}
\end{figure}
\end{center}

Finally we claim that the following pseudo-Anosov braid with the minimal dilatation $\delta(D_8)$ 
is not order-preserving.

\begin{lem}
\label{lem_mini8braid}
The braid  
$\delta_8^5 \sigma_7^{-1} \sigma_6^{-1} = (\sigma_1 \sigma_2 \sigma_3 \sigma_4 \sigma_5 \sigma_6 \sigma_7)^4 \sigma_1 \sigma_2 \sigma_3 \sigma_4 \sigma_5  \in B_8$ 
is not order-preserving. 
\end{lem}


\begin{proof}
Fix a basepoint of $\pi_1(D_8)$ in the interior of $D_8$. 
The following  automorphism corresponds to the braid 
 $\delta_8^5 \sigma_7^{-1} \sigma_6^{-1} $. 
 \begin{eqnarray*}
\phi:= \delta_8^5 \sigma_7^{-1} \sigma_6^{-1} : x_7 \mapsto x_2 \mapsto x_5 \mapsto x_6 \mapsto x_1 \mapsto x_4 \mapsto &x_6^{-1}x_8 x_6,&
\\
x_8 \mapsto x_3 \mapsto &x_6^{-1}x_7 x_6.&
\end{eqnarray*}
Assume that $\phi$ preserves a bi-ordering $<$. 
Without loss of generality, we may suppose that 
$x_7 < x_8$. 
Then $x_6^{-1}x_7 x_6 < x_6^{-1}x_8 x_6$ by the conjugation invariance. 
Since $\phi^{-1}$ preserves the same ordering $<$, 
we have 
$$x_8 = \phi^{-2} (x_6^{-1}x_7 x_6)< x_1 = \phi^{-2}( x_6^{-1}x_8 x_6).$$
From $x_7< x_8$ and $x_8< x_1$, we get $x_7<x_1$. 
Hence the ordering is increasing on the orbit of $x_7$: 
$$x_7 <  x_2 < x_5 < x_6 < x_1 < x_4 < x_6^{-1}x_8 x_6< \phi(x_6^{-1}x_8 x_6)< \cdots$$ 
In particular $x_6 <  x_6^{-1}x_8 x_6$, and this implies that 
$x_6< x_8$ by the conjugation invariance. 
Then we get $\phi^2(x_6)= x_4< \phi^2(x_8)= x_6^{-1}x_7 x_6$. 
On the other hand, $x_7 < x_6$ implies that 
$x_6^{-1}x_7 x_6 < x_6$ by the conjugation invariance again. 
But this is a contradiction, since 
$$x_6 < x_1 < x_4 < x_6^{-1}x_7 x_6 (< x_6).$$
Thus we conclude that 
$\phi$ does not preserve any bi-ordering of $F_8$. 
\end{proof}

\subsection{Questions}

Consider the braid $\beta_{m,n} \in B_{m+n+1}$ for $m,n \ge 1$ given by 
$$\beta_{m,n}= \sigma_1^{-1} \sigma_2^{-1} \cdots \sigma_m^{-1} \sigma_{m+1} \sigma_{m+2} \cdots \sigma_{m+n},$$ 
see Figure~\ref{fig_b-s_braid}(1). 
For example $\beta_{1,2}= \sigma_1^{-1} \sigma_2 \sigma_3$. 
The link $\mathrm{br}(\beta_{m,n})$ is equivalent to a $2$-bridge link as shown in Figure~\ref{fig_b-s_braid}(2).  
For any $m,n \ge 1$, the braid $\beta_{m,n}$ is pseudo-Anosov, see \cite[Theorem~3.9]{HK06}. 
We proved in Corollary~\ref{cor_beta-1m} that 
$\beta_{1,n}= \sigma_1^{-1} \sigma_2  \cdots \sigma_{n+1}$ 
is not order-preserving for each $n \ge 1$. 
We conjecture that 
$\beta_{m,n}$ is  not order-preserving for any $m,n \ge 1$. 

Let us introduce the braid $\gamma_{m,n} \in B_{m+n+3}$ for $m,n \ge 0$. 
If $m,n \ge 1$, 
 $\gamma_{m,n}$ is of the form 
$$\gamma_{m,n} = \sigma_1^{-2} \sigma_2^{-1} \cdots \sigma_{m+1}^{-1} \sigma_{m+2} \cdots \sigma_{m+n+1} \sigma_{m+n+2}^2.$$
For example $\gamma_{1,1}=  \sigma_1^{-2} \sigma_2^{-1} \sigma_3 \sigma_4^2 \in B_5$. 
If we ignore the $1$st string and the last string of $\gamma_{m,n}$, 
the resulting braid is an $(m+n+1)$-strand braid which is equal to $\beta_{m,n}$. 
We let 
$\gamma_{0,0}= \sigma_1^{-2} \sigma_2^2 \in P_3$, 
$\gamma_{0,1} = \sigma_1^{-2} \sigma_2 \sigma_3^2 \in B_4$ (see Remark~\ref{rem_disktwistC4}), and 
$\gamma_{1,0}= \sigma_1^{-2} \sigma_2^{-1} \sigma_3^2 \in B_4$.

One can prove in the same manner as in Theroem~\ref{thm_C4} and Remark~\ref{rem_disktwistC4} that 
 for each $m,n \ge 0$ 
$$S^3 \setminus \mathrm{br}(\gamma_{m,n}) \simeq S^3 \setminus \mathrm{br}(\sigma_1^{-2} \sigma_2^2) \simeq S^3 \setminus C_4.$$ 
Since $\sigma_1^{-2} \sigma_2^2 $ is a pseudo-Anosov, pure $3$-strand braid 
(in other words, $S^3 \setminus C_4$ is a hyperbolic $3$-manifolds whose fundamental group if bi-orderable),  
we have the following.

\begin{lem}
The braid  
$\gamma_{m,n}$ is pseudo-Anosov and order-preserving for $m,n \ge 0$. 
\end{lem}

If $(m,n) \ne (0,0)$, braids $\gamma_{m,n}$'s are non-pure and order-preserving. 
By Proposition~\ref{sym not id},  the braid $\gamma_{m,n}$ preserves some bi-ordering which is not standard ordering of the free group. 
Notice that for $m,n \ge 0$, 
the permutation of $\mathcal{S}_{m+n+3}$ associated to $\gamma_{m,n}$ has more than $1$ cycle, 
i.e, the closure of the braid $\gamma_{m,n}$ is a link with more than $1$ component. 
We ask the following.

\begin{ques}
Does there exist an order-preserving braid $\beta \in B_n$ ($n \ge 3$) whose permutation is cyclic 
(i.e, the closure $\widehat{\beta}$ is a knot)?
\end{ques}

\begin{center}
\begin{figure}
\includegraphics[width=4.8in]{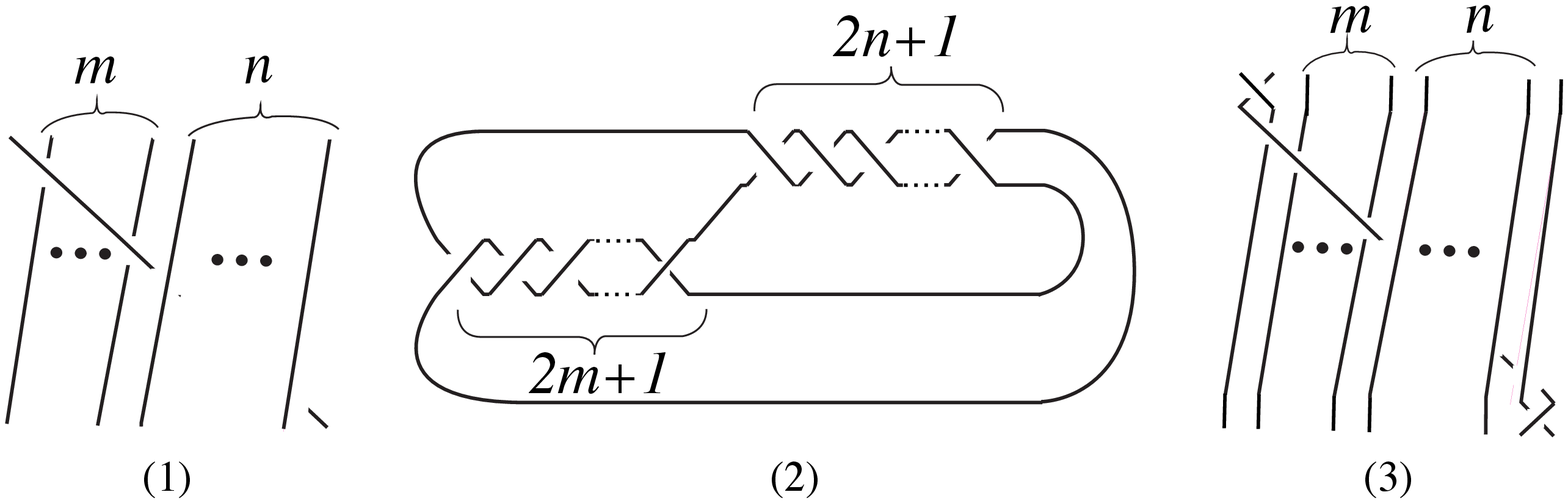}
\caption{(1) $\beta_{m,n} \in B_{m+n+1}$. (2) $2$-bridge link which is equivalent to $\mathrm{br}(\beta_{m,n})$. 
(3) $\gamma_{m,n} \in B_{m+n+3}$.}
\label{fig_b-s_braid}
\end{figure}
\end{center}

\appendix
\section{Ordering free groups}\label{Ordering free groups}

For any group $G$ the lower central series $G = \c_1G \supset \c_2G \supset \cdots$ is defined inductively by the formula $\c_{k+1}G := [G, \c_{k}G]$.  It is well known (see for example \cite{MKS76}) that if $F$ is a free group (or hyperbolic surface group) then the lower central quotients $\c_kF / \c_{k+1}F$ are free abelian of finite rank and also that $\cap_{k=1}^\infty \c_kF = \{1\}$. 
Following \cite{Neumann49} one can bi-order $F$ as we will now describe.  Choose an arbitrary bi-order $<_k$ of the (free abelian) group $\c_kF / \c_{k+1}F$ and define the positive cone for an ordering $<$ of $F$ as follows.  For $1 \ne g \in F$, declare $g$ to be positive if $1 <_k [g]$ in $\c_kF / \c_{k+1}F$, where $k$ is the unique integer for which $g \in \c_kF \setminus \c_{k+1}F$.
It is routine to check that this defines a bi-ordering $<$ of $F$.  We shall say that an ordering of $F$ defined in this way is a {\em standard} ordering of $F$.  If the rank of $F$ is greater than one, there are uncountably many standard orderings.  However there are uncountably many non-standard orderings of $F$ as well, for example as in the end of Section \ref{explicit}. 

\begin{prop}\label{LC convex}
For any standard ordering of the free group $F$, all the lower central subgroups $\c_kF$ are convex.
\end{prop}

\begin{proof}
Using invariance under multiplication it is sufficient to suppose $1 < f < g$ and $g \in \c_kF$, and show that $f \in \c_kF$.  Now $f \in \c_jF \setminus \c_{j+1}F$ for some unique positive integer $j$.  Assume for contradiction that $j < k$.  By the definition of $<$ we have that $1 <_j [f]$ in $\c_jF / \c_{j+1}F$.  But since $j < k$ we also have $g \in \c_jF$ and we see that 
$f^{-1}g \in \c_jF \setminus \c_{j+1}F$.  But since $1 < f^{-1}g$ we have $1 <_j [f^{-1}g]$ in $\c_jF / \c_{j+1}F$.  However, $g \in \c_{j+1}$ and so $[f^{-1}g] = [f^{-1}]$, which implies the contradiction that $1 <_j [f^{-1}]$.  Therefore we conclude that $j \ge k$ and therefore $f \in \c_jF \subset \c_kF$.  
\end{proof}

Note that for any automorphism of a group $\phi: G \to G$ we have $\phi(\c_kG) = \c_kG$ and so there are induced automorphisms $\phi_k: \c_kG/\c_{k+1}G \to \c_kG/\c_{k+1}G$.  In particular, 
$\phi_1$ is the same as the abelianization $\phi_{ab} : G/[G,G] \to G/[G,G]$.  The following is a well-known group theoretic result, but we include a proof, suggested by Thomas Koberda, for the reader's convenience.  

\begin{lem}
\label{lem_identity-map}
Suppose $\phi: G \to G$ is an automorphism of a group $G$ such that 
the abelianization $\phi_{ab} : G/[G,G] \to G/[G,G]$ is the identity map.  Then for each $k \ge 1$, the homomorphism $\phi_k: \c_kG/\c_{k+1}G \to \c_kG/\c_{k+1}G$ is also the identity map.
\end{lem}

\begin{proof}  We prove $\phi_k = id$ by induction on $k$, the hypothesis being the base case.  Note that the hypothesis implies that, for every $g \in G$ we have $\phi(g) = gc$ where $c \in \c_2G$ may depend upon $g$.  Now, suppose that $\phi_k: \c_kG/\c_{k+1}G \to \c_kG/\c_{k+1}G$ is the identity.
This implies that for $h \in \c_kG$ we have that $\phi(h) = hd$, where $d \in \c_{k+1}G$, also dependent upon $h$.

To show that $\phi_{k+1} = id$, consider any nonidentity element of $\c_{k+1}G$.  Such an element is a product of commutators of the form $[g,h]$ with $g \in G, h \in \c_kG$, so it suffices to show that
$$\phi_{k+1}([g,h]) \equiv [g,h]  \mod {\c_{n+2}G}.$$
We calculate $\phi_{k+1}([g,h]) \equiv [\phi(g),\phi(h)] = [gc, hd] = gchdc^{-1}g^{-1}d^{-1}h^{-1} \equiv 
 gchc^{-1}g^{-1}h^{-1}$, the latter equivalence following because $d \in \c_{k+1}G$ commutes with every element of $G$, modulo $\c_{k+2}G$. But note that 
 $gchc^{-1}g^{-1}h^{-1} = [g,h]x$, where $x = hg(h^{-1}chc^{-1})g^{-1}h^{-1}$.  Observe that the expression in parentheses is a commutator of $h^{-1} \in \c_kG$ and $c \in \c_2G$.  It is well-known 
 (see for example \cite[page~293]{MKS76}) that $[\c_mG,\c_nG] \subset \c_{m+n}G$.  Therefore, 
$h^{-1}chc^{-1} \in \c_{k+2}G$. Being a conjugate, $x$ also belongs to $\c_{k+2}G$, and therefore 
$\phi([g,h]) \equiv
[g,h]x \equiv [g,h] \mod {\c_{k+2}G}.$
\end{proof}

\begin{prop}\label{auto identity}
Suppose $\phi: F_n \to F_n$ is such that $\phi_{ab}: \Z^n \to \Z^n$ is the identity map $\phi_{ab} = id: \Z^n \to \Z^n$.  Then $\phi$ preserves {\em every} standard bi-ordering of $F_n$.
\end{prop}

\begin{proof}
By Lemma~\ref{lem_identity-map}, with $G = F_n$ we see that $\phi_k$ is the identity automorphism of 
$\c_kF_n/\c_{k+1}F_n$ and therefore preserves {\em every} ordering of $\c_kF_n/\c_{k+1}F_n$.  It follows that $\phi$ preserves the positive cone of every standard ordering of $F_n$.
\end{proof}

\section{The Whitehead link}
\label{Whitehead}

Let $W$ be the Whitehead link in $S^3$, see Figure~\ref{fig_2Whitehead}(1). 
Our goal is to prove: 

\begin{thm}
\label{thm_Whitehead}
The fundamental group of the Whitehead link complement ${\Bbb W} = S^3 \setminus W$  is bi-orderable. 
\end{thm}

\begin{center}
\begin{figure}
\includegraphics[width=2.5in]{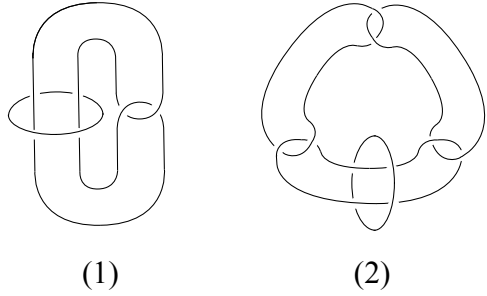}
\caption{(1) Whitehead link $W$. 
(2) ${\Bbb W}_3$ is homeomorphic to the complement of this link.}
\label{fig_2Whitehead}
\end{figure}
\end{center}

We first recall the Murasugi sum of surfaces. 
See \cite[Section~4.2]{Kawauchi96} for more details. 
Let $R_1$, $R_2$ and $R$ be compact oriented surface embedded in $S^3$. 
We say that $R$ is a ($2n$-){\it Murasugi sum} of $R_1$ and $R_2$ 
if we have the following. 
\begin{enumerate}
\item[(1)] 
$R= R_1 \cup R_2$, and $R_1 \cap R_2$ is a disk $D$ 
which satisfies the following. 
\begin{enumerate}
\item[(1.1)] 
$\partial D$ is a $2n$-gon with edges $a_1, b_1, a_2, b_2, \cdots, a_n, b_n$ 
enumerated in this order. 

\item[(1.2)] 
$a_i \subset \partial R_1$, and $a_i$ is a proper arc in $R_2$ for all $i$. 

\item[(1.3)] 
$b_i \subset \partial R_2$, and $b_i$ is a proper arc in $R_1$ for all $i$. 
\end{enumerate}

\item[(2)] 
There exist $3$-balls ${\Bbb B}_1$ and ${\Bbb B}_2$ in $S^3$ such that 
\begin{enumerate}
\item[(2.1)] 
${\Bbb B}_1 \cup {\Bbb B}_2 = S^3$, 
${\Bbb B}_1 \cap {\Bbb B}_2 = \partial {\Bbb B}_1 = \partial {\Bbb B}_2 = S^2$. 

\item[(2.2)] 
${\Bbb B}_i \supset R_i$ for $i= 1,2$. 

\item[(2.3)] 
$\partial {\Bbb B}_1 \cap R_1 = \partial {\Bbb B}_2 \cap R_2= D$.
\end{enumerate}
\end{enumerate}

In this paper we only use $2$-, $4$-Murasugi sums, see Figures~\ref{fig_2-MurasugiSum} and \ref{fig_M-sum}. 
A $2$-Murasugi sum corresponds to a connected sum of links.  
A $4$-Murasugi sum is a so-called  {\it plumbing}. 
To state the next theorem, 
we let $L_i = \partial R_i$ for $i= 1,2$ and $L= \partial R$ which are oriented links.

\begin{thm}[Theorem~1.3 and Corollary~1.4  in Gabai~\cite{Gabai85}] 
\label{thm_Gabai}
Suppose that $R$ is a Murasugi sum of $R_1$ and $R_2$. 
Then $L$ is a fibered link with a fiber $R$ if and only if 
$L_i$ is a fibered link with a fiber $R_i$ for $i= 1,2$. 
\end{thm}

Let $R_i$, $R$ and $L_i$ be as in Theorem~\ref{thm_Gabai}, and 
let $f_i: R_i \rightarrow R_i$ denote the monodromy. 
We may assume that $f_i|_{\partial R_i}$ equals the identity map $ id$. 
Let $R$ be a Murasugi sum of $R_1$ and $R_2$. 
By Theorem~\ref{thm_Gabai}, 
$L = \partial R$ is a fibered link with a fiber $R$ if $L_i$ is a fibered link with a fiber $R_i$ for $i= 1,2$. 
The following theorem tells us what the monodromy  $f: R \rightarrow R$ looks like. 

\begin{thm}[Corollary~1.4 in \cite{Gabai85}]
\label{thm_Gabai2}
The monodromy $f: R \rightarrow R$ is given by 
the product (i.e, the composition) $f= f_2'  f_1' : R \rightarrow R$, 
where $f_i'|_{R_i}$ equals $f_i$ and 
$f_i'|_{R \setminus R_i}$ equals  the identity map for $i= 1,2$.   
\end{thm}

\begin{convention}
\label{convention_product}
The product $f_2'  f_1'$ means that 
we first apply $f_2'$, then apply $f_1'$. 
\end{convention}


A {\it Hopf band} is an unknotted annulus in $S^3$. 
Two kinds of 
hopf bands $S_+$ and $S_-$ as in Figure~\ref{fig_Hopfband}(1) and (2) are called 
{\it positive} and ${\it negative}$ respectively. 
The links $L_+ = \partial S_+$ and $L_- = \partial S_-$ are called the {\it Hopf links}. 
It is known that 
$L_{\pm}$  is a fibered link with a fiber $S_{\pm}$. 
The monodromy $f_+ : S_+ \rightarrow S_+$ 
(resp. $f_- : S_- \rightarrow S_-$) 
 is the right handed Dehn twist (resp.  left handed Dehn twist) 
about the core circle of the annulus, see \cite[Figure~1]{GK90}.

\begin{center}
\begin{figure}
\includegraphics[width=4.8in]{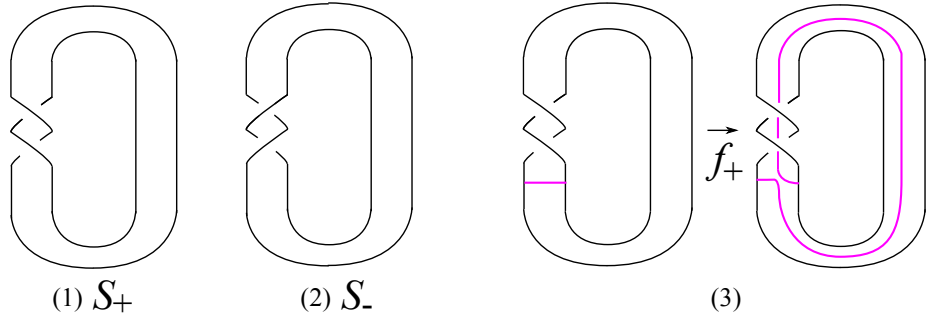}
\caption{Hopf bands (1) $S_+$ and (2) $S_-$. 
(3) Monodromy $f_+: S_+ \rightarrow S_+$ is the right handed Dehn twist about the core circle of  $S_+$. 
The figure illustrates the image of a proper arc (as shown in the left) 
under $f_+$.} 
\label{fig_Hopfband}
\end{figure}
\end{center}

\begin{center}
\begin{figure}
\includegraphics[width=4.6in]{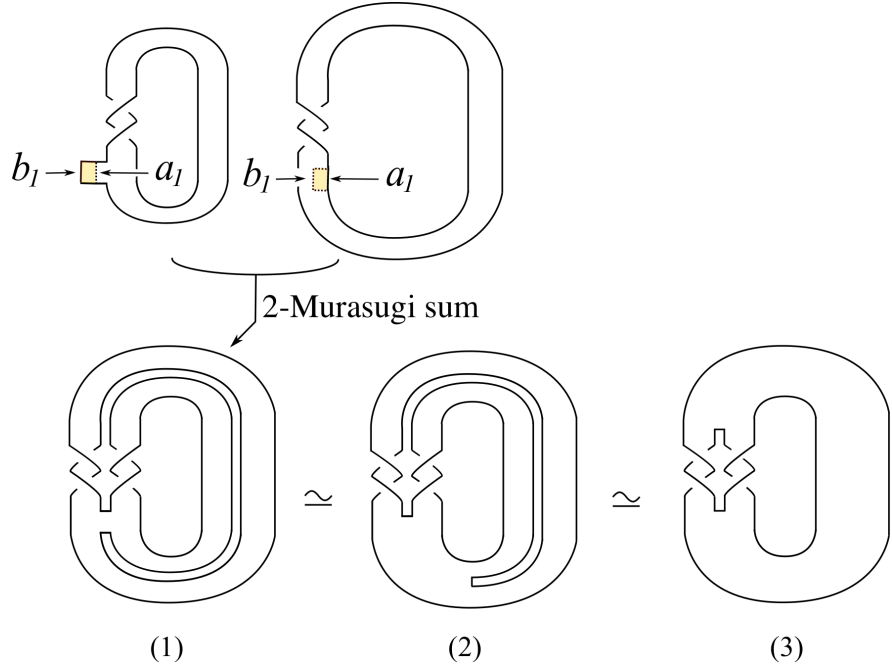}
\caption{(1) $2$-Murasugi sum $\hat{S}$ of  $S_+$ and  $S_-$. 
(The $2$-gon with edges $a_1$, $b_1$ is shaded.)
(2)(3) Surfaces which are isotopic to $\hat{S}$.}
\label{fig_2-MurasugiSum}
\end{figure}
\end{center}

\begin{center}
\begin{figure}
\includegraphics[width=5in]{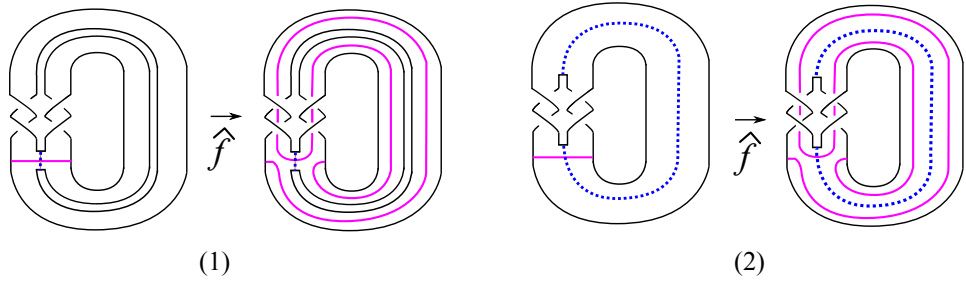}
\caption{(1) $\hat{f}= f_+' f_-':  \hat{S} \rightarrow \hat{S}$ using $\hat{S}$ in Figure~\ref{fig_2-MurasugiSum}(1).  
(2) $\hat{f}= f_+' f_-':  \hat{S} \rightarrow \hat{S}$ using $\hat{S}$ in Figure~\ref{fig_2-MurasugiSum}(3).  
Both (1) and (2) illustrate the images of solid and broken arcs  (as shown in the left) under $\hat{f}$.}
\label{fig_TwoHopfbands}
\end{figure}
\end{center}

\begin{proof}[Proof of Theorem~\ref{thm_Whitehead}]

We will find a fibered surface $S$ of the Whitehead link $W$ and its monodromy $f: S \rightarrow S$. 
Let $\hat{S}$ be the $2$-Murasugi sum of hopf bands $S_+$ and $S_-$, 
see Figure~\ref{fig_2-MurasugiSum}(1).  
Figure~\ref{fig_2-MurasugiSum}(3) is a surface which is isotopic to $\hat{S}$.  
By Theorem~\ref{thm_Gabai}, 
$\partial \hat{S}$ is a fibered link with a fiber $\hat{S}$. 
Theorem~\ref{thm_Gabai2} tells us that 
$f_+'  f_-': \hat{S} \rightarrow \hat{S}$ serves 
the monodromy $\hat{f}: \hat{S} \rightarrow \hat{S}$. 
Figure~\ref{fig_TwoHopfbands}(1)(2) shows the images of two proper arcs (solid and broken arcs) 
under $\hat{f}: \hat{S} \rightarrow \hat{S}$. 
This is the so-called {\it point-pushing map}, see \cite[Section~4.2]{FM12}.

A $4$-Murasugi sum of $\hat{S}$ and $S_+$ gives rise to a fibered surface $S$ of $W$, 
 which is a torus with $2$ boundary components, 
see Figure~\ref{fig_M-sum}. 
By Theorem~\ref{thm_Gabai2}, 
the product $  (\hat{f})' f_+': S \rightarrow S$ serves the monodromy $f: S \rightarrow S$. 
Note that $ (\hat{f})': S \rightarrow S$ is a pushing map along the arc $m$, 
and $f_+': S \rightarrow S$ is the right handed Dehn twist about a simple closed curve $\ell$, 
see Figure~\ref{fig_Simple}(1).

Shrinking  one of the boundary components to a puncture, 
one can take a simple model as a representative of $S$, which is a torus with one boundary component and with a puncture as in Figure~\ref{fig_Simple}(2).  
Abusing the notation, we denote such a simple model by the same notation $S$. 
We also denote 
the corresponding loop based at the puncture and the corresponding closed loop in the simple model by the same notations $m$ and $\ell$, 
see Figure~\ref{fig_Simple}(2).

\begin{center}
\begin{figure}
\includegraphics[width=3.5in]{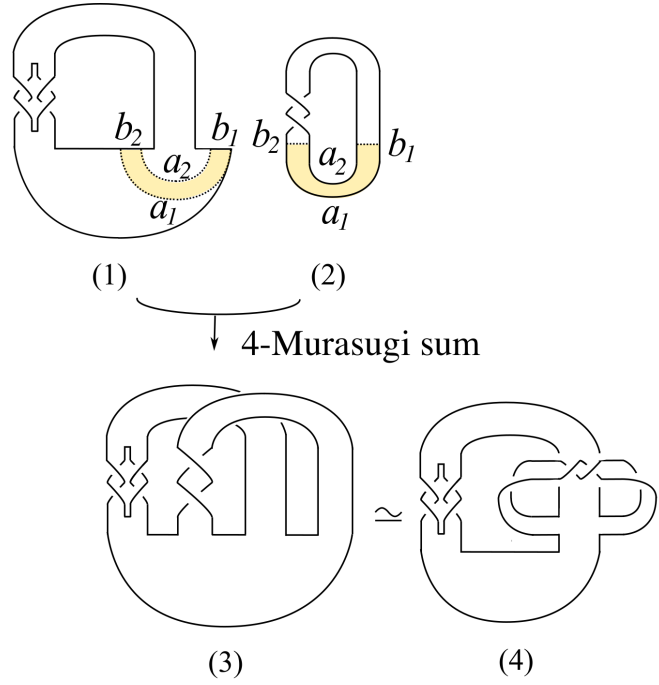}
\caption{(1) $\hat{S}$. (2) $S_+$. 
(3)(4) $4$-Murasugi sum of $\hat{S}$ and $S_+$ is a fibered surface $S$ of  $W$. 
(The $4$-gon with edges $a_1, b_1, a_2,b_2$  is shaded in the figures (1)(2).)}
\label{fig_M-sum}
\end{figure}
\end{center}

\begin{center}
\begin{figure}
\includegraphics[width=3in]{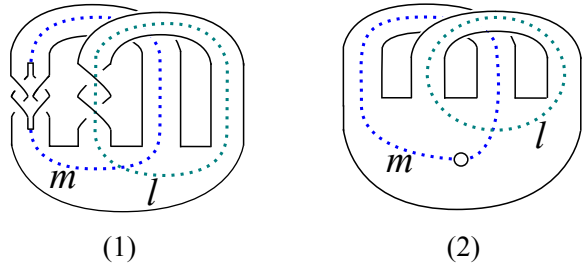}
\caption{(1) A fibered surface $S$ of $W$. 
(2) A simple model of $S$.}
\label{fig_Simple}
\end{figure}
\end{center}

\begin{center}
\begin{figure}
\includegraphics[width=4.5in]{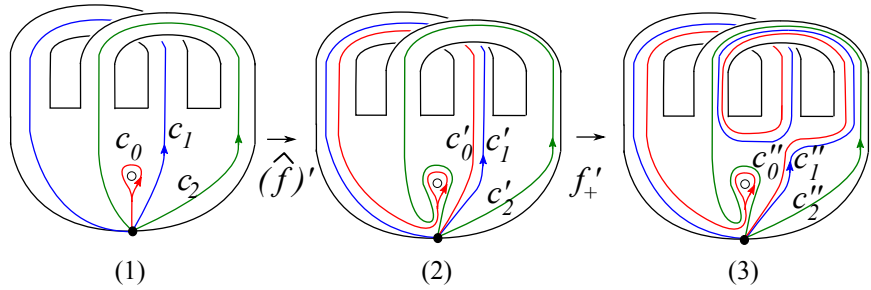}
\caption{(1) Loops $c_0$, $c_1$, $c_2$. 
(2) Images of $c_i$'s under $(\hat{f})'$, 
where $c_i':= (\hat{f})'(c_i)$. 
(3) Images of $c_i$'s under $f=  (\hat{f})'f_+'$, 
where $c_i'':= f(c_i)$.}
\label{fig_W-monodromy}
\end{figure}
\end{center}

We choose a basepoint $p$  on $\partial S$, 
and take oriented loops $c_0$, $c_1$ and $c_2$ based at this point 
so that $c_0$ is a loop surrounding the puncture, and 
$c_1$, $c_2$ are the meridian and the longitude of the torus, see Figure~\ref{fig_W-monodromy}(1). 
Abusing the notations, we denote the equivalence class of $c_i$ in $\pi_1(S,p) = F_3$ by the same notation $c_i$. 
(Then $\{c_0,c_1, c_2\}$ is a generating set of $\pi_1(S,p) $.) 
The images of $c_1$, $c_2$ and $c_3$ under  $f_* = ((\hat{f})' f_+')_*:  \pi_1(S,p) \rightarrow  \pi_1(S,p)$ are given as follows. 
See Convention~\ref{convention_product}.  
\begin{eqnarray*}
f_*(c_0) &=& c_2 c_0^{-1} c_1 c_0 c_1^{-1} c_0 c_2^{-1}, 
\\
f_*(c_1) &=& c_2 c_0^{-1} c_1, 
\\
f_*(c_2) &=& c_2 c_0^{-1}, 
\end{eqnarray*}
see Figure~\ref{fig_W-monodromy}(3) for $f_*$. 
Let us consider the abelianization $(f_*)_{\mathrm{ab}}: {\Bbb Z}^3 \rightarrow {\Bbb Z}^3$. 
From the above computation of $f_*(c_i)$, we have 
\begin{eqnarray*}
(f_*)_{\mathrm{ab}}[c_0] &=& \hspace{3mm} [c_0], 
\\
(f_*)_{\mathrm{ab}}[c_1] &=&  - [c_0]+ [c_1]+ [c_2], 
\\
(f_*)_{\mathrm{ab}}[c_2] &=&  - [c_0] \hspace{10mm} + [c_2].
\end{eqnarray*}
By calculation one sees that the characteristic polynomial of $(f_*)_{\mathrm{ab}}$ equals 
$(t-1)^3$, and 
all eigenvalues of  $(f_*)_{\mathrm{ab}}$ are $1$.  
By Theorem~\ref{Perron-Rolfsen}, it follows that $f_*$ is order-preserving. 
Note that  
$\pi_1({\Bbb W})= F_3 \rtimes_{f_*} {\Bbb Z}$. 
By Proposition~\ref{prop_fibration-HNN}, we conclude that $\pi_1({\Bbb W})$ is bi-orderable. 
\end{proof}



\bibliographystyle{amsplain}
\bibliography{braidautord}

\end{document}